\documentclass[10pt, reqno]{amsart}

\usepackage{amsmath,amssymb,amsthm,amsfonts,verbatim}
\usepackage{microtype}
\usepackage[all,2cell]{xy}
\usepackage{mathtools}
\usepackage{graphicx}
\usepackage{pinlabel}
\usepackage{hyperref}
\usepackage{mathrsfs}
\usepackage{color}
\usepackage[dvipsnames]{xcolor}
\usepackage{enumerate}
\usepackage{cite}
\usepackage{soul}

\CompileMatrices

\usepackage[top=1.2in,bottom=1.2in,left=1in,right=1in]{geometry}

\usepackage{hyperref} 
\hypersetup{          
	colorlinks=true, breaklinks, linkcolor=[RGB]{51 102 204}, filecolor=Orchid, urlcolor=[RGB]{51 102 204},
	citecolor=Orchid, linktoc=all, }
\usepackage[nameinlink]{cleveref}

\theoremstyle{plain}
\newtheorem{theorem}{Theorem}[section]
\newtheorem{maintheorem}{Theorem}

\newtheorem{proposition}[theorem]{Proposition}
\newtheorem{lemma}[theorem]{Lemma}

\newtheorem{corollary}[theorem]{Corollary}

\theoremstyle{definition}
\newtheorem{definition}[theorem]{Definition}
\newtheorem{example}[theorem]{Example}

\newtheorem{remark}[theorem]{Remark}

\newcommand{\nc}{\newcommand}
\nc{\dmo}{\DeclareMathOperator}

\nc{\Q}{\mathbb{Q}}
\nc{\F}{\mathbb{F}}
\nc{\R}{\mathbb{R}}
\nc{\Z}{\mathbb{Z}}
\nc{\C}{\mathbb{C}}
\nc{\N}{\mathbb{N}}
\nc{\Ell}{\mathcal{L}}
\nc{\M}{\mathcal{M}}
\nc{\K}{\mathcal{K}}
\nc{\I}{\mathcal{I}}
\nc{\T}{\mathcal T}
\nc{\U}{\mathcal U}
\nc{\disk}{\mathbb{D}}
\nc{\hyp}{\mathbb{H}}

\nc{\CP}{\mathbb{CP}}
\nc{\RP}{\mathbb{RP}}
\dmo{\Mod}{Mod}
\dmo{\PMod}{PMod}
\dmo{\LMod}{LMod}
\dmo{\Diff}{Diff}
\dmo{\Homeo}{Homeo}
\dmo{\dist}{dist}
\dmo\BDiff{BDiff}
\dmo\SO{SO}
\dmo\Hom{Hom}
\dmo\SL{SL}
\dmo\rank{rank}
\dmo\sig{sig}
\dmo\Out{Out}
\dmo\Aut{Aut}
\dmo\Inn{Inn}
\dmo\GL{GL}
\dmo\PGL{PGL}
\dmo\Gr{Gr}
\dmo\PSL{PSL}
\dmo\BHomeo{BHomeo}
\dmo\EHomeo{EHomeo}
\dmo\EDiff{EDiff}
\dmo\Disc{Disc}
\dmo\Aff{Aff}

\renewcommand{\bar}{\overline}

\dmo\Teich{Teich}
\dmo\Fix{Fix}
\nc{\pair}[1]{\ensuremath{\left\langle #1 \right\rangle}}
\nc{\abs}[1]{\ensuremath{\left| #1 \right|}}
\nc{\action}{\circlearrowright}
\nc{\norm}[1]{\left | \left | #1 \right | \right |}
\nc{\abcd}[4]{\ensuremath{\left(\begin{array}{cc} #1 & #2 \\ #3 & #4 \end{array}\right)}}
\nc{\into}\hookrightarrow
\dmo{\Isom}{Isom}
\nc{\normal}{\vartriangleleft}
\dmo{\Vol}{Vol}
\dmo{\im}{Im}
\dmo{\Push}{Push}
\dmo{\Conf}{Conf}
\dmo{\PConf}{PConf}
\dmo{\PB}{PB}
\dmo{\id}{id}
\dmo{\Jac}{Jac}
\dmo{\Pic}{Pic}
\dmo{\Stab}{Stab}
\dmo{\Arf}{Arf}
\dmo{\End}{End}
\dmo{\Gal}{Gal}
\dmo{\lcm}{lcm}
\dmo{\ab}{ab}
\dmo{\opp}{op}
\dmo{\SU}{SU}
\dmo{\OT}{\Omega \mathcal{T}}
\dmo{\OM}{\Omega \mathcal{M}}
\dmo{\PH}{\mathbb{P}\mathcal{H}}
\dmo{\spin}{spin}
\dmo{\even}{even}
\dmo{\odd}{odd}
\dmo{\comp}{\mathcal{H}}
\dmo{\Mgk}{\mathcal{M}_{g, \underline{\kappa}}}
\dmo{\orb}{orb}
\dmo{\AJ}{AJ}
\dmo{\Ck}{\mathsf{C}(\underline{\kappa})}
\dmo{\Int}{Int}
\dmo{\pr}{pr}
\dmo{\lab}{lab}
\dmo{\Sym}{Sym}
\dmo{\Ann}{Ann}
\dmo{\Rad}{Rad}
\dmo{\Ind}{Ind}
\dmo{\Div}{Div}
\dmo{\Res}{Res}
\dmo{\Hur}{Hur}
\dmo{\vcd}{vcd}

\nc{\Span}[1]{\operatorname{Span}(#1)}

\renewcommand{\epsilon}{\varepsilon}
\renewcommand{\tilde}{\widetilde}
\renewcommand{\le}{\leqslant}
\nc{\coloneq}{\mathrel{\mathop:}\mkern-1.2mu=}
\nc{\margin}[1]{\marginpar{\scriptsize #1}}
\nc{\para}[1]{\medskip\noindent\textbf{#1.}}
\definecolor{myblue}{RGB}{102,153, 255}
\definecolor{myred}{RGB}{204,0,0}
\definecolor{mygreen}{RGB}{0,204,0}
\definecolor{myorange}{RGB}{255,102,0}
\definecolor{mypurple}{RGB}{138,43,226}
\nc{\red}[1]{\textcolor{myred}{#1}}
\nc{\blue}[1]{\textcolor{myblue}{#1}}

\nc{\B}{\mathscr B}
\nc{\MDk}{\mathcal{MD}(\kappa)}
\dmo{\interior}{int}
\dmo{\sgn}{sgn}

\title{Stratified braid groups: monodromy}

\author{Nick Salter}
\email{nsalter@nd.edu}
\address{Department of Mathematics, University of Notre Dame, Hurley Hall, Notre Dame, IN 46556}
\date{April 14, 2023}

\begin{document}
\maketitle
\begin{abstract}
The space of monic squarefree complex polynomials has a stratification according to the multiplicities of the critical points. We introduce a method to study these strata by way of the infinite-area translation surface associated to the logarithmic derivative $df/f$ of the polynomial. We determine the monodromy of these strata in the braid group, thus describing which braidings of the roots are possible if the orders of the critical points are required to stay fixed. Mirroring the story for holomorphic differentials on higher-genus surfaces, we find the answer is governed by the framing of the punctured disk induced by the horizontal foliation on the translation surface. 
\end{abstract}

\section{Introduction}
Let $\Conf_n(\C)$ denote the space of unordered configurations of $n$ distinct points in $\C$. Equivalently, this is the space of monic squarefree polynomials of degree $n$. From this point of view, $\Conf_n(\C)$ carries a natural {\em stratification} $\{\Conf_n(\C)[\kappa]\}$ indexed by partitions $\kappa$ of $n-1$. A polynomial $f \in \Conf_n(\C)$ belongs to $\Conf_n(\C)[\kappa]$ if and only if the roots of $f'$ form the partition $\kappa$ of $n-1$. Thinking of $f$ as a mapping $f: \C \to \C$, the partition $\kappa$ describes the profile of the branched covering, i.e. the multiplicities of the critical points.

\para{Main Question} Understand the topology of $\Conf_n(\C)[\kappa]$. What is the fundamental group $\B_n[\kappa]:=\pi_1(\Conf_n(\C)[\kappa])$ (a ``stratified braid group'')? Is $\Conf_n(\C)[\kappa]$ a $K(\pi,1)$ space?\\

The strata $\Conf_n(\C)[\kappa]$ admit descriptions as certain discriminant complements, and so the Main Question is reminiscent of the conjecture of Arnol'd-Pham-Thom on the homotopy type of discriminant complements associated to isolated hypersurface singularities. As we will see, $\Conf_n(\C)[\kappa]$ is also closely related to a certain stratum $\Omega_\kappa$ of meromorphic differentials on $\CP^1$, and so the Main Question is a version of the conjecture posed by Kontsevich--Zorich originally for strata of holomorphic differentials on higher-genus curves \cite{KZpre}. 

In this article, we begin this project by answering a natural question about the groups $\B_n[\kappa]$. The inclusion map $\Conf_n(\C)[\kappa] \into \Conf_n(\C)$ induces a {\em monodromy homomorphism} into the braid group $B_n := \pi_1(\Conf_n(\C))$
\[
\bar \rho: \B_n[\kappa] \to B_n,
\]
and we describe the image $B_n[\kappa]:=\bar \rho(\B_n[\kappa])$. We find (cf. \Cref{section:wnf}) that there is a {\em crossed homomorphism} $\phi_\kappa: B_n \to (\Z/r\Z)^r$, where $r = \gcd(\kappa)$ is the gcd of the parts of $\kappa$, which characterizes the monodromy image when $n \ge n_0(r)$ is sufficiently large compared to $r$. The precise bounds $n_0(r)$ are somewhat intricate and are stated in \Cref{section:generating}, but we note here that if $r = 1$ then $n_0(1) = 2$ (in which case $\phi_\kappa$ is trivial and $B_n[\kappa] = B_n$) and if $r = 2$ then $n_0(2) = 8$; in general, $n_0(r) \le 14 r$ in the worst-case scenario.

\begin{maintheorem}\label{mainthm:monodromy}
    For any $n \ge 2$ and any partition $\kappa$ of $n-1$ with $\gcd(\kappa) = r$, there is a containment
    \[
    B_n[\kappa] \le \ker(\phi_\kappa).
    \]
    If $n \ge n_0(r)$, then this containment is an equality. 
\end{maintheorem}
\begin{remark}
\Cref{theorem:gammakernel} gives a simple characterization of $\ker(\phi_\kappa)$ (and hence $B_n[\kappa]$) in the range $n \ge n_0(r)$: it is the subgroup of $B_n$ consisting of braids admitting representatives where at each overcrossing, there are $r$ strands passing underneath. 
\end{remark}

\begin{remark} The containment $B_n[\kappa] \le \ker(\phi_\kappa)$ is {\em not} always an equality. For $\kappa = \{n-1\}$, the fundamental group $\pi_1(\Conf_n(\C)[\kappa])$ is cyclic, since the only polynomials in this stratum are of the form $f(z) = (z-z_0)^n - c$, while the corresponding $\ker(\phi_\kappa)$ has finite index in $B_n$. On the other hand, the braid-theoretic methods underlying the proof of \Cref{mainthm:monodromy} are almost certainly not optimized, and it would be interesting to know exactly how small $n_0(r)$ can be taken.
\end{remark}

\begin{remark}[The braided Gauss-Lucas theorem]
The classical Gauss-Lucas theorem asserts that the critical points lie inside the convex hull of the roots. There is a refinement $\rho: \B_n[\kappa] \to B_{n, \abs{\kappa}}$ of $\bar \rho$ where one tracks both roots and critical points, and the study of $\rho$ can be viewed as a Gauss-Lucas theorem for {\em families} of polynomials. The classical theorem implies that every braid in the image of $\rho$ admits a representative where at each time $t$, the critical points lie in the convex hull of the roots. Our study of the monodromy shows that {\em this is not sufficient}. \Cref{figure:nonreal} in \Cref{subsection:braidedGL} gives an example of a braid satisfying this convexity condition which is not realizable as the braid of root and critical points of any family of polynomials. \Cref{mainthm:monodromy} shows that when $r \ge 2$, there are even certain braidings of the roots alone (e.g. a half-twist) which cannot be realized by polynomial families. We plan to return to a study of the refined monodromy $\rho$ in future work.
\end{remark}

\para{From polynomials to translation surfaces} Our method of study is built around a type of uniformization map. Namely, we associate to $f \in \Conf_n(\C)[\kappa]$ its logarithmic derivative $df/f$. This is a meromorphic differential on $\CP^1$, with $n$ simple poles of residue $2 \pi i$ at the zeroes of $f$ and an additional simple pole of residue $-2\pi i n$ at infinity. Such an object can be viewed as a translation surface - the poles give the surface infinite area, but it nevertheless has a very simple global structure (see, e.g. \Cref{figure:sd1}). Let $\Omega_\kappa$ denote the moduli space of meromorphic differentials on $\CP^1$ with $n$ simple poles of residue $2 \pi i$, a simple pole at infinity (necessarily of residue $-2\pi i n$), and zeroes of multiplicity specified by $\kappa$. Some elementary complex analysis (\Cref{lemma:corresp}) shows that every such differential is of the form $df/f$ for $f \in \Conf_n(\C)[\kappa]$. The assignment $f \mapsto df/f$ therefore gives a classifying map 
\[
\mu: \Conf_n(\C)[\kappa] \to \Omega_\kappa.
\]
It is clear that if $f(z)$ and $g(z)$ are related by an affine change of variables $g(z) = f(az+b)$, then the associated differentials $df/f$ and $dg/g$ determine the same point in $\Omega_\kappa$. The converse is not much harder, but this identification is fundamental to our approach, and we record it here for good measure.

\begin{theorem}\label{mainthm:iso}
    The classifying map $\mu$ induces an isomorphism of complex orbifolds
    \[
    \Conf_n(\C)[\kappa]/\Aff \cong \Omega_\kappa.
    \]
\end{theorem}

The advantage in studying $\Omega_\kappa$ is that its global structure is much more apparent. The equations defining $\Conf_n(\C)[\kappa]$ as a discriminant complement are highly nonlinear, and it is difficult to construct and analyze the behavior of explicit loops inside $\Conf_n(\C)[\kappa]$. On the other hand, in $\Omega_\kappa$ the corresponding analysis is elementary via deformations of the associated translation surfaces. Moreover, in \Cref{proposition:cellstruct}, we use the combinatorics of the translation surfaces to obtain an explicit finite cell structure on $\Omega_\kappa$, in principle reducing the study of the topology of $\Omega_\kappa$ to the combinatorics of the ``labeling systems'' that index the cells.\\

The meromorphic differential $df/f$, and more precisely its incarnation as an infinite-area translation surface, plays a fundamental role in our analysis of the monodromy. \Cref{mainthm:monodromy} exactly parallels a result in the setting of strata of holomorphic differentials on higher-genus surfaces. Here, the problem is to determine the image of the orbifold fundamental group in the mapping class group of the surface. This was answered in \cite{strata3}, where it is shown that the image is essentially characterized by the property that the monodromy must preserve the {\em framing} of the surface (punctured at the locations of the zeros) associated to the horizontal vector field specified by the translation surface structure. In the case where the locations of the zeroes are not marked, the monodromy must preserve a certain distillate of the framing known as an {\em $r$-spin structure}, c.f. \cite{strata2}. Here, we find the exact same sort of characterization of the monodromy: the crossed homomorphism $\phi_\kappa$ measures a ``change in winding number'' of arcs relative to the framing of the punctured surface induced from $df/f$. In both of these settings, the integer $r$ is given as the gcd of the orders of the zeroes of the differential.

\para{Related work} The space $\Omega_\kappa$ fits into the theory of the ``isoresidual fibration'' studied by Gendron--Tahar \cite{gt,gt2}. They consider the map from the space of meromorphic differentials with prescribed zero and pole orders to the vector space of residues, showing among other things that in the case of a single zero, the map is a fibration away from a hyperplane arrangement. Our $\Omega_\kappa$ is the fiber of the isoresidual map of differentials on $\C$ over the vector of residues $(2 \pi i, \dots, 2\pi i)$, where the zeroes have order specified by $\kappa$.

In \cite{mcd}, Dougherty--McCammond investigate various combinatorial structures induced from polynomial maps. One of their key tools is a pair of transverse singular foliations on $\C$ with singularities at the zeroes and critical points of $f$. We obtain an equivalent pair of foliations from the horizontal and vertical foliations of the translation surface structure on $\C$ induced by $df/f$. See \Cref{remark:dm}.

\para{Outline} In \Cref{section:uniformization}, we establish \Cref{mainthm:iso}, showing that one can study polynomials in a stratum $\Conf_n(\C)[\kappa]$ by instead studying the translation surfaces associated to their logarithmic derivatives. In \Cref{section:tslsurf}, we describe the structure of an individual $df/f$ as a translation surface, as well as the global structure of the stratum $\Omega_\kappa$. Our main results here are the discussion in \Cref{subsection:stripdecomp} of the ``strip decomposition'' of $df/f$, and the global structure theorem \Cref{proposition:cellstruct}, which exhibits a cell structure on $\Omega_\kappa$ coming from the combinatorics of this decomposition. The proof of \Cref{mainthm:monodromy} is carried out in \Cref{section:wnf,section:construct,section:generating}. In \Cref{section:wnf}, we show how the translation surface structure associated to $df/f$ constrains the monodromy image $B_n[\kappa]$, forcing it to preserve {\em winding numbers} of arcs on the disk. In \Cref{section:construct}, we exhibit certain loops in $\Omega_\kappa$ and analyze their monodromies in $B_n$. Finally in \Cref{section:generating}, we show that this finite collection of elements is enough to generate the kernel of $\phi_\kappa$ (when $n$ is sufficiently large). The key tool here is to relate the winding number crossed homomorphism $\phi_\kappa$ to an {\em a priori} totally different crossed homomorphism $\Upsilon_r$ formulated in terms of a count of ``virtual undercrossings'' on a braid diagram, and then to establish a {\em factorization algorithm} (\Cref{lemma:factorpure}) for expressing the kernel of $\Upsilon_r$ in terms of elements known to lie in the monodromy image.

\para{Acknowledgements} The author would like to thank Tara Brendle and Matt Day for interesting discussions, and Dan Margalit for very helpful feedback and for alerting the author to the work \cite{mcd} of Dougherty--McCammond. The author is supported by NSF Award No. DMS-2153879.

\section{Moduli spaces of polynomials, differentials, and translation surfaces}\label{section:uniformization}

We begin with a discussion of the space $\Omega_\kappa$, the stratum of translation surfaces associated to the differentials $df/f$. We construct this here as a moduli space, by taking a quotient of the space of differentials by the relevant automorphism group. The main result of this section is \Cref{mainthm:iso}, recorded here as \Cref{prop:corresp}, which amounts to little more than an unpacking of the definitions, but lays the foundation for what is to follow, as it will allow us to explore the space $\Conf_n(\C)[\kappa]$ by instead exploring the space $\Omega_\kappa$ of translation surfaces.

Let $\kappa = \{k_1, \dots, k_p\}$ be a partition of $n-1$. Here and throughout, we write $\abs{\kappa} = p$ to denote the number of parts of the partition. Let $\MDk$ be the set of meromorphic differential forms $\omega$ on $\CP^1$ satisfying the following properties:
\begin{itemize}
    \item There are exactly $p$ zeroes of $\omega$ of orders $k_1, \dots, k_p$, and each zero lies in $\C \subset \CP^1$,
    \item There are $n$ simple poles each of residue $2 \pi i$ contained in $\C$, and an additional simple pole at $\infty$ of residue $-2\pi i n$.
\end{itemize}

The following is basic complex analysis; we include the argument for the sake of completeness. 
\begin{lemma}\label{lemma:corresp}
Let $\omega \in \MDk$ be given. Then there is a unique $f \in \Conf_n(\C)[\kappa]$ such that $\omega = \frac{df}{f}$.
\end{lemma}
\begin{proof}
Let $f \in \Conf_n(\C)$ be the polynomial with simple roots at the $n$ poles $z_1, \dots, z_n$ of $\omega$ contained in $\C$. By the theory of partial fractions,
\[
\frac{df}{f} = \left(\frac{1}{z-z_1} + \dots + \frac{1}{z-z_n}\right) dz,
\]
on $\C$, showing that $\omega - \frac{df}{f}$ has no poles on $\C$. By hypothesis, $\omega$ and $\frac{df}{f}$ have simple poles at $\infty$ of equal residue, so that $\omega - \frac{df}{f}$ is moreover holomorphic in a neighborhood of $\infty$. Thus $\omega - \frac{df}{f}$ is a {\em holomorphic} differential form on $\CP^1$; the only such form is $0$ (see, e.g. \cite[Exercise IV.1.A]{miranda}).
\end{proof}

Observe that the {\em affine group} 
\[
\Aff = \{\alpha \in \Aut(\CP^1) \mid \alpha (\infty) = \infty\} = \{az+b \mid a \in \C^*,\ b \in \C\} \cong \C^*\ltimes \C
\]
acts via biholomorphisms on $\MDk$ on the left via inverse-pullback: 
\[
\alpha \cdot \omega = (\alpha^{-1})^*(\omega).
\]
Likewise, there is a left action of $\Aff$ on $\Conf_n(\C) \subset \C^n$ induced from the diagonal action on $\C^n$. 

Note that for each of these actions, $\Aff$ is a Lie group acting properly by holomorphic automorphisms with finite stabilizers. The orbit spaces $\Conf_n(\C)[\kappa]/\Aff$ and $\mathcal{MD}(\kappa)/\Aff$ therefore carry complex orbifold structures. We observe that $\dim_\C(\Conf_n(\C)[\kappa]) = \abs{\kappa}+1$, since the roots of $f'$ move in a configuration space of dimension $\abs{\kappa}$, and the generic antiderivative of $f'$ has $n$ distinct roots. Thus $\dim_\C(\Conf_n(\C)[\kappa]/\Aff) = \abs{\kappa}-1$.

We define the {\em $\kappa$-stratum of logarithmic derivatives} as the second of the orbifolds discussed above:
\[
\Omega_\kappa := \mathcal{MD}(\kappa)/\Aff.
\]
Observe that there is a natural map
\begin{align*}
\mu: \Conf_n(\C)[\kappa] &\to \MDk\\
    f &\mapsto \frac{df}{f}.
\end{align*}

\begin{proposition}[\Cref{mainthm:iso}]
\label{prop:corresp}
The map $\mu$ is an $\Aff$-equivariant biholomorphism, inducing an isomorphism of complex orbifolds of dimension $\abs{\kappa}-1$.
\[
\Conf_n(\C)[\kappa]_{\Aff} \cong \Omega_\kappa.
\]
\end{proposition}
\begin{proof}
    That $\mu$ is a bijection follows immediately from \Cref{lemma:corresp}, and it is easy to see that this respects the complex structures on the domain and codomain. Equivariance is also easily verified, as
    \[
    (\alpha^{-1})^* \frac{df}{f} = \frac{d (f\circ \alpha^{-1})}{f \circ \alpha^{-1}},
    \]
     and the polynomial $f \circ \alpha^{-1}$ has simple roots at the points $\alpha(z_1), \dots, \alpha(z_n)$, where $z_1, \dots, z_n$ are the roots of $f$. 
    \end{proof}


    
\para{An exact sequence} To conclude this section, we study the relationship between the (orbifold) fundamental groups of $\Conf_n(\C)[\kappa]$ and its quotient $\Conf_n(\C)[\kappa]/\Aff \cong \Omega_\kappa$. Following the discussion in \cite[Introduction]{looijenga}, we find that there is an exact sequence
\[
\pi_1(\Aff) \to \pi_1(\Conf_n(\C)[\kappa]) \to \pi_1^{orb}(\Conf_n(\C)/{\Aff}) \to \pi_0(\Aff) \to 1.
\]
Recalling that $\pi_1(\Aff) = \Z$ and $\pi_0(\Aff) = 1$, and also recalling that $\pi_1(\Conf_n(\C)[\kappa]):= \B_n[\kappa]$, we obtain the exact sequence
\begin{equation}
    \label{equation:les}
    \Z \to \B_n[\kappa] \to \pi_1^{orb}(\Omega_\kappa) \to 1.
\end{equation}
In particular, we emphasize that the projection $\B_n[\kappa] \to \pi_1^{orb}(\Omega_\kappa)$ is surjective. It is not hard to show that \eqref{equation:les} is in fact short exact, but we do not need this fact here so we will not elaborate.

\section{$\Omega_\kappa$ as a space of translation surfaces}\label{section:tslsurf}

The purpose of this section is to explain the structure of a differential $df/f$ when realized as a translation surface. In \Cref{subsection:foliation}, we begin with a discussion of some generalities of translation surfaces induced by meromorphic differentials and the induced horizontal foliation. In \Cref{subsection:stripdecomp}, we discuss the notion of a {\em strip decomposition} of the translation surface for $df/f$ and some important related notions ({\em strips, slits, fixed/free prongs}). This will give a combinatorial decomposition of $\Omega_\kappa$ into cells; in \Cref{subsection:global}, we discuss the global structure of this decomposition. In the body of this paper, we will only make use of the {\em constructive} aspects of the theory we establish here (as a technique for exploring the space $\Omega_\kappa$ and computing the monodromy of loops); in later work, we hope to make use of the global structure theory obtained in \Cref{proposition:cellstruct}. 

\subsection{Flat cone metrics and the horizontal foliation}\label{subsection:foliation} The integration map
\[
z \mapsto \int_{z_0}^z \frac{df}{f}
\]
provides a system of holomorphic charts on $\C$ away from the zeroes of $f$ and $f'$ for which the transition functions are translations $z \mapsto z+c$. The horizontal foliation on $\C$ given by lines of constant real part (equivalently determined as the kernel of the real $1$-form $dy = \im(dz)$) pulls back to a singular foliation $\mathcal F$ on the domain. 

Near a zero $\zeta_i$ of order $k_i$, these charts realize $\zeta_i$ as a cone point with cone angle $2\pi(k_i + 1)$. At such a point $\mathcal F$ has a prong singularity of order $2k_i+2$. In the flat coordinates, the prongs alternate between pointing to the right and left and will be referred to as such. We also note that there is a natural cyclic ordering on both the left and right prongs, and each set of prongs carries the structure of a torsor over $\Z/(k_i+1)\Z$ by measuring the counterclockwise angle from one prong to the other.

The local structure near a simple pole is slightly less well-known, but is equally straightforward. First note that in the case of $\omega = dz/z$, the integration map (i.e. the logarithm) sends the punctured disk $0<\abs{z} \le 1$ to the half-infinite strip 
\[
S^L= \{a+bi \mid a \le 0,\ 0 \le b \le 2 \pi\}
\]
with the top and bottom identified via the translation $z \mapsto z + 2 \pi i$; likewise, for $c \in \C^*$, the integration map sends a neighborhood of $c\, dz/z$ near $0$ to the rotated strip $c S^L$. For a general differential $\omega = g(z) dz /z$ with a simple pole at $z = 0$, the coordinate 
\[
w(z) = \exp(\int_{z_0}^z \omega)
\]
pulls $\omega$ back to $dw/w$, showing that in general, a neighborhood of a simple pole of residue $c$ is realized on the translation surface via the rotated strip $c S^L$ (again with opposite edges identified). At such a pole, $\mathcal F$ has an ``infinite prong singularity'', where the foliation structure is locally given by the set of rays emanating from a point.

Given a translation surface $T$ represented as a finite collection of disjoint polygons $\{P_i\}\subset \C$ with edge identifications, a {\em cut move} is a subdivision of some $P_i$ into $P_i^1, P_i^2$ along with the identification of the cut edges. To perform a {\em paste move}, take distinct polygons $P_i, P_j$ for which there is an edge of $P_i$ identified to an edge of $P_j$, and translate $P_j$ so that the identified edges coincide. If $P_i$ and $P_j$ overlap only along this edge, then the paste move can be performed by joining $P_i$ and the translate of $P_j$ into a single polygon, inheriting the remaining edge identifications. It is a basic fact in the theory of translation surfaces that $T$ and $T'$ determine the same point in their stratum (in this case, $\Omega_\kappa$) if and only if they are related by a sequence of cut/paste moves.

\para{The horizontal foliation for $\boldmath{df/f}$} The integration map $\int df/f$ induces a translation surface structure on an $n$-times punctured plane, for which the horizontal foliation has the local features discussed above. Conversely, any ``combinatorially suitable'' translation surface structure $T$ on an $n$-times punctured plane determines a differential $\omega = \frac{df}{f} \in \Omega_\kappa$. Here, by ``combinatorially suitable'', we mean the following:
\begin{itemize}
    \item $T$ has $n$ half-infinite cylindrical strips $S_1^L, \dots, S_n^L$, each equivalent to $S^L$ via a translation (each extends infinitely far to the left and has height $2 \pi i$),
    \item $T$ has one half-infinite cylindrical strip $S_\infty$ equivalent to $-n S^L$ via a translation (thus extending infinitely far to the right and of height $2 n \pi i$),
    \item $T$ has cone points $p_1, \dots, p_m$ of orders $k_1, \dots, k_m$, where $\kappa = \{k_1, \dots, k_m\}$,
    \item The complement of the strips $S_1^L, \dots, S_n^L, S_\infty$ has finite area.
\end{itemize}
That every such translation surface is induced by a differential $df/f \in \Omega_\kappa$ is immediate: $T$ induces a Riemann surface structure on an $n$-times punctured plane, equipped with a differential $\omega$ on which $S_1^L, \dots, S_n^L$ correspond to simple poles of residue $2 \pi i$, $S_\infty$ corresponds to a simple pole of residue $-2n \pi i$, and which has zeroes of multiplicity specified by $\kappa$; i.e. $\omega \in \Omega_\kappa$.\\

The global structure of the horizontal foliation $\mathcal F$ on $\CP^1$ induced by $df/f$ is extremely simple.

\begin{lemma}
    \label{lemma:horizfol}
    Let $f \in \Conf_n(\C)[\kappa]$ be given, and let $\mathcal F$ be the horizontal singular foliation on $\CP^1$ induced by $df/f$. Then $\mathcal F$ has the following properties:
    \begin{itemize}
        \item With the finitely many exceptions of leaves incident to a critical point of $f$, every leaf connects a zero of $f$ to $\infty$. In particular, $\mathcal F$ has no closed leaves.
        \item Let $p_j$ be a critical point of order $k_j$, corresponding to a cone point of order $k_j$ on the translation surface and inducing a $2 k_j + 2$-pronged singularity of $\mathcal F$. Then the prongs alternate between terminating at a zero of $f$ and at $\infty$. In particular, at most one prong at $p_j$ terminates at each zero of $f$.
    \end{itemize}
    \end{lemma}
\begin{proof}
    We first claim that $\mathcal F$ has no closed leaves. Integration of $df/f$ along such a leaf would yield a real period of $df/f$, but the periods of $df/f$ are purely imaginary. If a leaf does not terminate at a singularity, it must accumulate somewhere on the compact space $\CP^1$. Such a nearly-closed leaf can be completed via a short vertical segment into a simple closed curve whose period has positive real part, again a contradiction. Thus every leaf must terminate at both ends at a singularity of $\mathcal F$. At a critical point of order $k$, $\mathcal F$ has a $2k+2$-pronged singularity, so that there are finitely many leaves terminating at a critical point as claimed. Integrating $df/f$ along a path terminating at a zero of $f$ has real part tending to $-\infty$, so that at most one end of every leaf can terminate at such a point; likewise at most one end can terminate at $\infty$. 
    
    This same observation proves the second assertion: when integrating along consecutive prongs, the real part of $\int df/f$ is {\em monotonic}, so that exactly one prong in each consecutive pair terminates at a zero of $f$. If two prongs at $p_j$ terminate at the same zero, we consider the bounded region of the plane enclosed by these leaves. By the above, there must be at least one prong originating inside this region which must terminate at $\infty$, but it cannot escape the region enclosed by the two prong leaves, showing a contradiction.
\end{proof}

\begin{remark}\label{remark:dm}
In \cite{mcd}, Dougherty and McCammond study a pair of transverse singular foliations equivalent to those induced by the real and imaginary parts of $df/f$. Their point of view is somewhat different: they induce $\mathcal F$ by pulling back the transverse foliations $\abs{z} = c$ and $\arg(z) = c$ on $\C^*$ under the map $f: \C \to \C$, but the result as unmeasured foliations is the same. They equip their foliations with measures that are different from the ones coming here from the flat structure, considering instead the measure induced by the Euclidean structure on $\C^*$. Using this, they are able to obtain a detailed picture of various combinatorial structures associated to the polynomial $f$. It would be interesting to see if the translation surface perspective has anything to add to the story they pursue.
\end{remark}

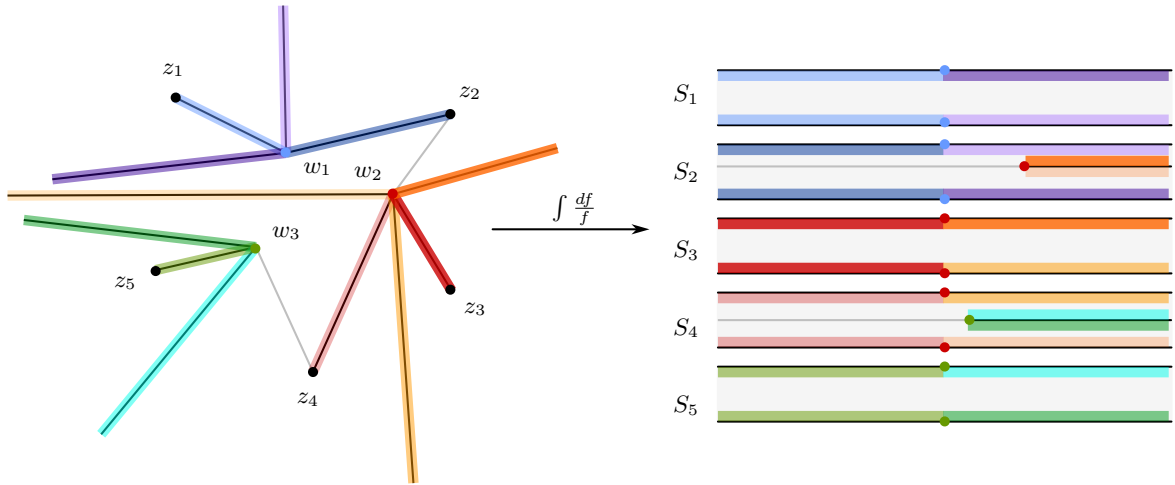
\begin{figure}[ht]
\labellist
\small
\pinlabel $z_1$ [b] at 90.70 218.25
\pinlabel $z_2$ [bl] at 243.76 206.91
\pinlabel $z_3$ [tl] at 246.59 102.04
\pinlabel $z_4$ [t] at 164.39 53.85
\pinlabel $z_5$ [tr] at 73.69 116.21
\pinlabel $w_1$ [tl] at 158.73 178.57
\pinlabel $w_2$ [br] at 208.33 162.98
\pinlabel $w_3$ [bl] at 141.72 130.38
\pinlabel $S_1$ [r] at 382.64 212.58
\pinlabel $S_2$ [r] at 382.64 170.06
\pinlabel $S_3$ [r] at 382.64 127.55
\pinlabel $S_4$ [r] at 382.64 85.03
\pinlabel $S_5$ [r] at 382.64 42.52
\pinlabel $\int\frac{df}{f}$ at 310 150
\endlabellist
\includegraphics[scale=0.7]{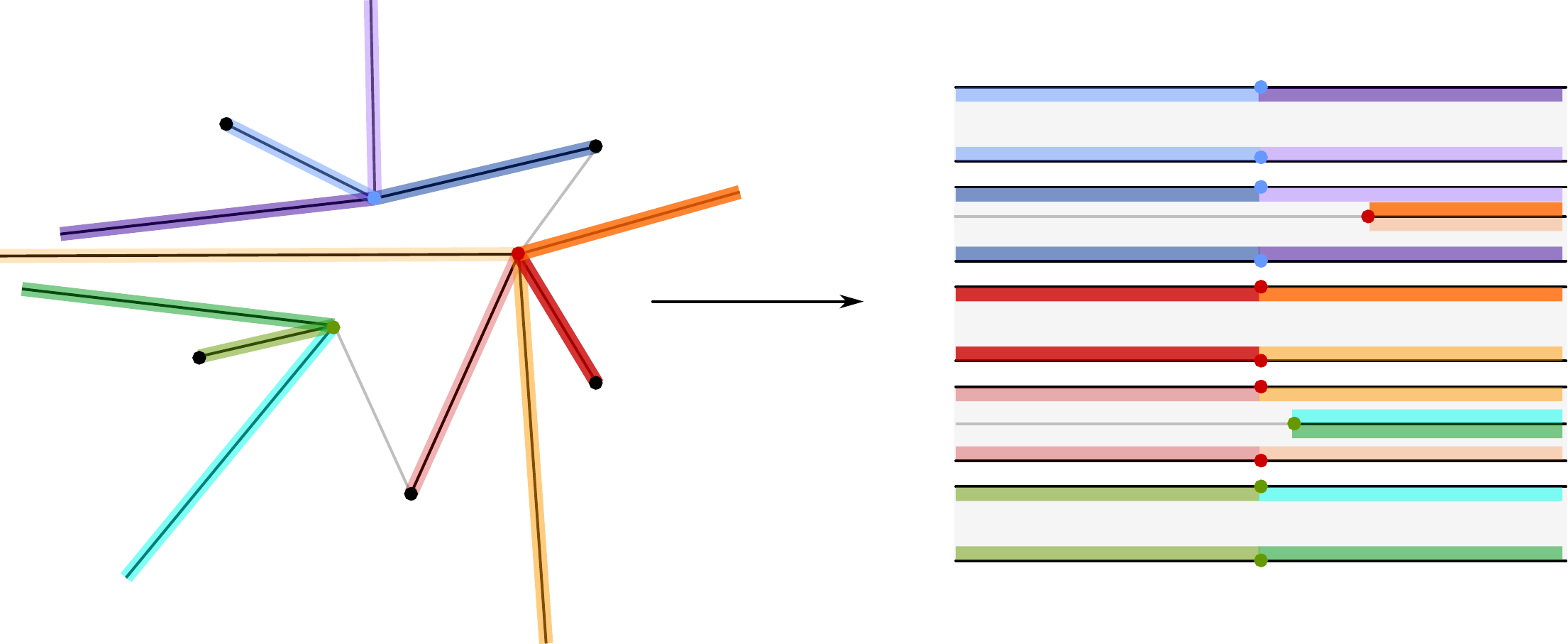}
\caption{A strip decomposition of a differential $df/f$. The polynomial $f$ has simple zeroes $z_1, \dots, z_5$, and critical points $w_1, w_2, w_3$ with $w_2$ having multiplicity $2$. On the left, the points $z_i, w_j$ are shown in $\C$ along with the prongs of the horizontal foliation. The strip decomposition is shown on the right. Colors on the right indicate gluing instructions, and correspond to the colorings of the prongs at left.}
\label{figure:sd1}
\end{figure}

\subsection{Strip decomposition}\label{subsection:stripdecomp} A translation surface $T \in \Omega_\kappa$ admits a finite number of combinatorially-determined standard forms which we call a {\em strip decomposition}. Assign a numbering $S_1^L, \dots, S_n^L$ to the $n$ left-infinite strips of height $2 \pi i$, or equivalently a numbering $z_1, \dots, z_n$ of the zeroes of $f$. The strips $S_i^L$ extend infinitely far to the left by hypothesis. Following the leaves of the horizontal foliation in $S_i^L$ to the right (towards $\infty$), we observe that there must be at least one leaf terminating at a critical point $w_j$, for otherwise, this region would close up into a topological cylinder, rendering the translation surface disconnected (except, of course, in the case $n = 1$ with differential $dz/z$). Choosing one such leaf, we fix an identification $S_i^L \cong S^L$ by identifying $w_j$ with $0 \sim 2 \pi i$ in $S^L$. We remark that we allow for the non-generic possibility that the leaf connecting $z_i$ to $w_j$ passes through one or more additional cone point.

Define $S_i \subset T$ as the continuation to the right of the leaves of the horizontal foliation passing through $S_i^L$. This is then a bi-infinite strip of height $2 \pi i$, possibly containing additional cone points. The boundary of $S_i^L$ is determined by the prong of $w_j$ from which the leaf terminating at $z_i$ emanates. We say that $S_i$ is {\em bounded} by the cone point $w_j$, and call the distinguished prong a {\em fixed prong}. By \Cref{lemma:horizfol}, all of the leaves of the horizontal foliation not terminating at a cone point are contained in some strip $S_i$, and so this produces a decomposition of $T$ as claimed. 

Strips $S_i$ and $S_j$ are said to be {\em vertically adjacent} if the top right boundary of $S_i$ is identified with the lower right boundary of $S_j$ or vice versa. We will speak of the strips {\em above} and {\em below} $S_i$ via this definition.

\para{Strip coordinates} The strip decomposition of $T$ of course depends on various non-canonical choices. To track this, and moreover to understand the global structure of the space $\Omega_\kappa$, we define
\[
\Omega_\kappa^{ord} \to \Omega_\kappa
\]
as the covering space consisting of differentials $df/f \in \Omega_\kappa$ together with labelings $z_1, \dots, z_n$ and $w_1, \dots, w_p$ of the zeroes and critical points of $f$, respectively. Note that the stabilizer of a {\em labeled} configuration of two or more points in $\C$ under the affine group is trivial, so that $\Omega_\kappa^{ord}$ is a manifold cover of the orbifold $\Omega_\kappa$. For the ensuing discussion, we will lift $df/f \in \Omega_\kappa$ to one of its preimages in $\Omega_\kappa^{ord}$.

Having fixed such data, one can then encode the combinatorial type of a strip decomposition by tracking the prongs of the cone points. A cone point $w_j$ of order $k_j$ has $2k_j+2$ prongs emanating from it, of which $k_j+1$ point to the left on the translation surface. Since $\sum_{j=1}^p k_j = n-1$, there is a total of $n-1+p$ left prongs. Of these, $n$ are fixed prongs; we call the remaining $p-1$ {\em free prongs}. Generically, the leaf emanating from a free prong is contained in the interior of a unique strip $S_i$; exceptionally it may terminate at some other cone point. Given a labeled differential $df/f$, a choice of strip decomposition yields the following data:

\begin{enumerate}
    \item For each zero $z_i$ of $f$, a choice of some left prong to bound the strip $S_i$,
    \item An assignment of the remaining $p-1$ free prongs to one of the strips containing it (generically, a free prong lies in a unique strip; exceptionally it may lie on the boundary between two that are vertically adjacent),
    \item The $p-1$ relative periods $\gamma_j$ of the arcs connecting each free prong to the fixed prong for its strip, each an element of $\R + [0,2\pi] i$; if two free prongs belong to the same strip, the relative periods must be distinct (so that the cone points do not collide).
\end{enumerate}

Conversely, we can use the relative periods of the free prongs to put a system of coordinates (``strip coordinates'') on $\Omega_\kappa^{ord}$. A strip coordinate chart is indexed by a {\em labeling system} which includes the data specified by (1) and (2) above. Without further constraint, the relative periods of (3) do not yet induce a coordinate patch on $\Omega_\kappa^{ord}$: as two free prongs in the same strip orbit around one another, one will pass through the slit associated to the other and into a different strip. To prevent this, we additionally impose {\em orderings} on the imaginary parts of the relative periods of the free prongs within a given strip.

\begin{definition}[Labeling system]
    \label{definition:labels}
    Fix a partition $\kappa$ of $n-1$ and consider the associated set of left prongs $\mathcal P_\kappa$ of cardinality $n+p-1$. A {\em labeling system} $L$ is a choice of the following data:
    \begin{enumerate}
        \item For each $1 \le i\le n$, a choice of prong $v\in \mathcal P_\kappa$ as the fixed prong for $S_i$,
        \item An assignment of each of the remaining $p-1$ prongs in $\mathcal P_\kappa$ to some strip $S_i$,
        \item For each strip $S_i$, a choice of {\em ordering} $v_1 \le \dots \le v_{m_i}$ of the $m_i$ free prongs assigned to $S_i$.
    \end{enumerate}
\end{definition}

Not every labeling system is realized by some $df/f \in \Omega_\kappa^{ord}$, since some choices of labeling systems will cause the translation surface to be {\em disconnected}. Here we {\em state} the combinatorial criterion for connectedness only; we prove that this encodes topological connectedness in \Cref{lemma:stripintoomega}.

\begin{definition}[Connected labeling system]
    \label{definition:connectedLS}
    Let $L$ be a labeling system for some partition $\kappa = \{k_1, \dots, k_p\}$ of $n-1$. Let $\Gamma_L$ be the graph whose vertices are the parts $k_i$ of $\kappa$, and where $k_i$ and $k_j$ are connected by an edge if there are prongs of $k_i$ and $k_j$ contained in the same strip $S_i$. Then $L$ is said to be {\em connected} if $\Gamma_L$ is. 
\end{definition}

Having specified a labeling system, we turn now to the problem of parametrizing the relative periods of the free prongs. For $m \ge 1$, define the closed $m$-simplex via
\[
\Delta^m = \{ (y_1, \dots, y_m) \in \R^m \mid 0 \le y_1 \le \dots \le y_m \le 2 \pi\}.
\]

\begin{definition}[Strip coordinate domain]
    \label{definition:stripcoord}
    Let $L$ be a labeling system of some partition $\kappa$ of $n-1$; for $1 \le i \le n$, suppose there are $m_i$ free prongs assigned to the strip $S_i$. The associated {\em strip coordinate domain} is the set
    \[
    \Omega_L := \left(\prod_{i = 1}^n \R^{m_i} + i \Delta^{m_i}\right) \setminus \mathcal D \subset \C^{p-1},
    \]
    where $\mathcal D$ is the union of the following sets:
    \begin{enumerate}
        \item $\mathcal D_1$ the set of points where $z_{j_1} = z_{j_2}$ for $j_1, j_2$ assigned to the same strip,
        \item $\mathcal D_2$ the set of points where $x_{j_1} = x_{j_2} < 0$ and $y_{j_2} = y_{j_1} + 2 \pi i$ for $j_1, j_2$ assigned to the same strip,
        \item $\mathcal D_3$ the set of points where $x_j = x_k > 0$ and $y_j = 2 \pi i, y_k= 0$, where $j$ is assigned to the strip below the strip containing $k$,
        \item $\mathcal D_4$ the set of points where $z_j = 0$ or $z_j = 2 \pi i$.
    \end{enumerate}
\end{definition}



\begin{lemma}
    \label{lemma:stripintoomega}
    Let $L$ be a connected labeling system of the partition $\kappa$. Then there is a realization map
    \[
    r: \Omega_L \to \Omega_\kappa^{ord}.
    \]
    The restriction of $r$ to the interior of $\Omega_L$ is a biholomorphism onto its image. 
\end{lemma}

\begin{proof}
    A point $\gamma = (\gamma_1, \dots, \gamma_{p-1}) \in \Omega_L$ determines a translation surface $T_\gamma$ as follows: assemble $n$ bi-infinite strips $S_1, \dots, S_n$, and mark $0 \sim 2 \pi i \in S_i$ with the prong specified by $L$. Given a relative period $\gamma_j$, the labeling system specifies a free prong in a strip $S_i$; place the free prong at $\gamma_j \in S_i$ and introduce a slit running horizontally to the right from $\gamma_j$ to $\infty$. The cyclic ordering on the prongs at each cone point then specifies gluing instructions on the slits as well as on the right halves of the top and bottom boundary components of each $S_i$ (the top and bottom left halves of $S_i$ are identified to each other). The excision of the set $\mathcal D$ from $\Omega_L$ ensures that after gluing, no pair of free prongs are identified, and that no free prong is placed at the location of a fixed prong. 

    We claim that $T_\gamma$ is connected if and only if the labeling system $L$ is connected in the sense of \Cref{definition:connectedLS}. Note that the vertices of $\Gamma_L$ are canonically identified with the cone points $w_j$ of $T_\gamma$. A first trivial observation is that $T_\gamma$ will be connected if and only if there is a path connecting each pair of strips $S_i, S_j$. Suppose that $T_\gamma$ is connected; we wish to find a path in $\Gamma_L$ connecting an arbitrary pair of vertices $w_j, w_k$. Choose prongs at $w_j$ and at $w_k$; these live in strips $S, S'$ respectively. If $S = S'$ then $w_j$ and $w_k$ are connected by definition; otherwise, let $c$ be a path in $T_\gamma$ connecting $S$ to $S'$, and one can use this to build a corresponding path in $\Gamma_L$ by moving through a sequence of cone points lying in the sequence of strips passed through by $c$.
    
    Conversely, suppose that $\Gamma_L$ is connected. Given a strip $S_i$, let $w_i$ be the corresponding bounding cone point. Observe that it suffices to show that strips $S_a$ and $S_b$ are path-connected in $T_\gamma$ if the associated vertices $w_a, w_b$ of $\Gamma_L$ are either equal or adjacent. If $w_a=w_b = w$ then a path connecting $S_a$ to $S_b$ can be constructed by winding some number of times around $w$. If $w_a$ and $w_b$ are adjacent, then there is some strip $S_c$ containing a prong of both $w_a$ and $w_b$; a path connecting $S_a$ to $S_b$ can be constructed by concatenating a path between neighborhoods of $w_a$ and $w_b$ in $S_c$ with paths winding around $w_a$ and $w_b$ between $S_c$ and $S_a$, resp. $S_b$.

    At this point, we have shown that this construction process yields a well-defined map $r: \Omega_L \to \Omega_\kappa^{ord}$. We next observe that $r$ is holomorphic - this is a simple consequence of the fact that the relative period maps on $\Omega_\kappa^{ord}$ are holomorphic. It remains to show that $r$ is injective on the interior of $\Omega_L$. To see this, observe that translation surfaces $T_\gamma$ and $T_\delta$ determine the same point in $\Omega_\kappa^{ord}$ only if the sets of relative periods between cone points are equal. The sets of relative periods between fixed cone points is a torsor on the group of absolute periods; in this case the absolute periods is just the set $2 \pi i \Z \subset \C$. If $\gamma, \delta \in \interior(\Omega_L)$ are distinct, then the corresponding relative periods all have imaginary part strictly between $0$ and $2\pi$, so that the relative periods of $T_\gamma$ cannot be obtained from those of $T_\delta$ by translation by some absolute period. 
    \end{proof}

\subsection{Change of coordinates; a cell structure on $\boldmath{\Omega_\kappa^{ord}}$}\label{subsection:global} We next consider the transition maps between strip coordinate domains with overlapping image. There are two basic transitions to study: (1) changing which of the prongs in $S_i$ is fixed, and (2) pushing the topmost (relative to the ordering) free prong out the top right side of $S_i$ and into the bottom of the strip above (or in reverse, pushing the bottom free prong through the bottom right side). All coordinate changes are compositions of these two, e.g. pushing a free prong out the top left side is equivalent to changing the free prong to the fixed, and pushing the new free prong (formerly the fixed) out the bottom right. The lemmas below record the effects of these moves on strip coordinates; the proofs follow from inspection of \Cref{figure:cc1,figure:cc2}.

\begin{lemma}[Type 1: changing the fixed prong]
    \label{lemma:fixedtofree}
    Let $L$ be a connected labeling system for $\kappa$. Choose some strip $S_i$; let $v_0$ denote the fixed prong and let $v_1 \le \dots \le v_{m_i}$ denote the free prongs in $S_i$ together with their cyclic ordering. Define $L'$ as the labeling system obtained from $L$ by choosing some $v_j$ as the new fixed prong for $S_i$, and ordering the free prongs via 
    \[
    v_{j+1} \le \dots \le v_{m_i} \le v_0 \le v_1 \le \dots \le v_{j-1}.
    \]
    The map on relative periods $t: \Omega_L \to \Omega_{L'}$ is given by
    \[
    t(\gamma_1, \dots, \gamma_{m_i}) = (\gamma_{j+1} - \gamma_j, \dots,\gamma_{m_i} - \gamma_j, 2\pi i- \gamma_j, \gamma_1 + 2\pi i - \gamma_j, \dots, \gamma_{j-1}+ 2 \pi i-\gamma_j).
    \]
\end{lemma}
\begin{figure}[ht]
\labellist
\small
\pinlabel $v_0$ [bl] at 107.71 5.67
\pinlabel $v_1$ [br] at 58.85 32.85
\pinlabel $v_2$ [tr] at 117.38 49.35
\pinlabel $v_3$ [br] at 175.73 65.19
\pinlabel $v_2$ [bl] at 388.31 8.50
\pinlabel $v_3$ [br] at 448.00 17
\pinlabel $v_0$ [br] at 375.30 32.85
\pinlabel $v_1$ [br] at 329.95 62.36
\endlabellist
\includegraphics[scale=0.9]{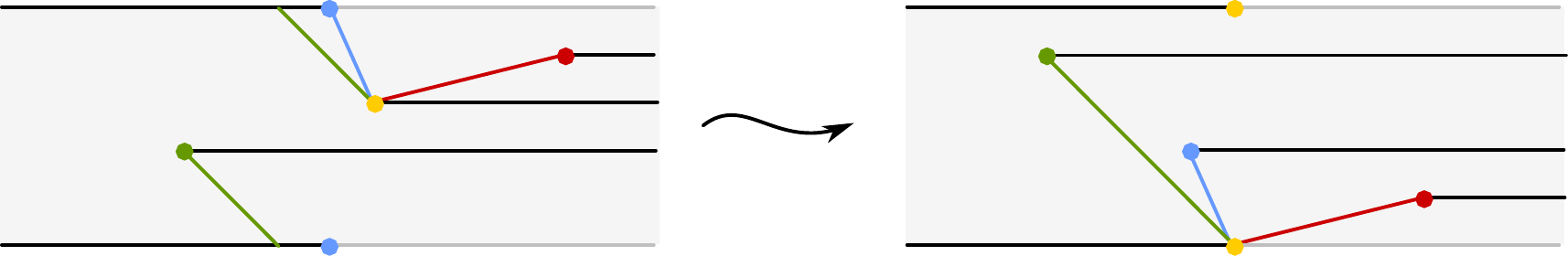}
\caption{Type 1: changing the fixed prong from $v_0$ to $v_2$.}
\label{figure:cc1}
\end{figure}
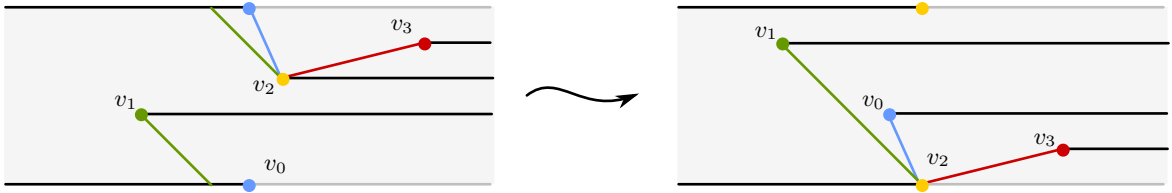

\begin{lemma}[Type 2: pushing up/down]
\label{lemma:pushup}
Let $L$ be a connected labeling system for $\kappa$. Let $S_i$ be a strip with fixed prong $v_0$ and let $v_1 \le \dots \le v_{m_i}$ denote the free prongs in $S_i$ together with their cyclic ordering. Denote the relative periods by $\gamma_1, \dots, \gamma_{m_i}$, and suppose that $\gamma_{m_i} = x + 2 \pi i$ with $x > 0$. Let $S_j$ be the strip whose bottom right boundary is identified with the top right boundary of $S_i$, and let $v_1'\le \dots \le v_{m_j}'$ denote the free prongs in $S_j$. 

Changing the assignment of $v_{m_i}$ from $S_i$ to $S_j$ yields the labeling system $L'$ obtained from $L$ by reassigning $v_{m_i}$ to $S_j$ with ordering in $S_j$
\[
v_{m_i}\le v_1' \le \dots \le v_{m_j}'
\]
and relative period $x$.

Conversely, if $v_1$ has period $x \in \R$ with $x > 0$, then we may reassign it to the strip $S_j$ whose top right boundary is identified with the bottom right on $S_i$, assigning it to the maximal position in $S_j$ with period $x + 2 \pi i$. 
\end{lemma}

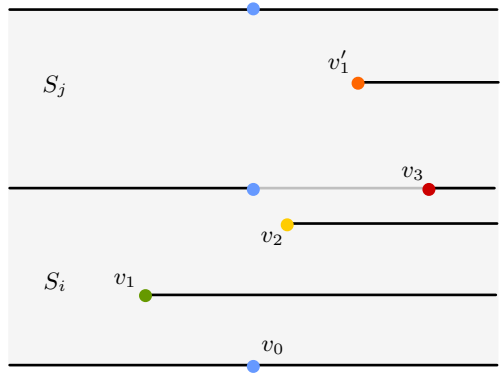
\begin{figure}[ht]
\labellist
\small
\pinlabel $S_i$ [b] at 19.84 31.18
\pinlabel $S_j$ [b] at 19.84 113.38
\pinlabel $v_0$ [bl] at 104.87 5.67
\pinlabel $v_1$ [br] at 56.69 34.01
\pinlabel $v_2$ [tr] at 119.04 62.19
\pinlabel $v_3$ [br] at 178.57 79.36
\pinlabel $v_1'$ [br] at 147.39 124.71
\endlabellist
\includegraphics[scale=0.9]{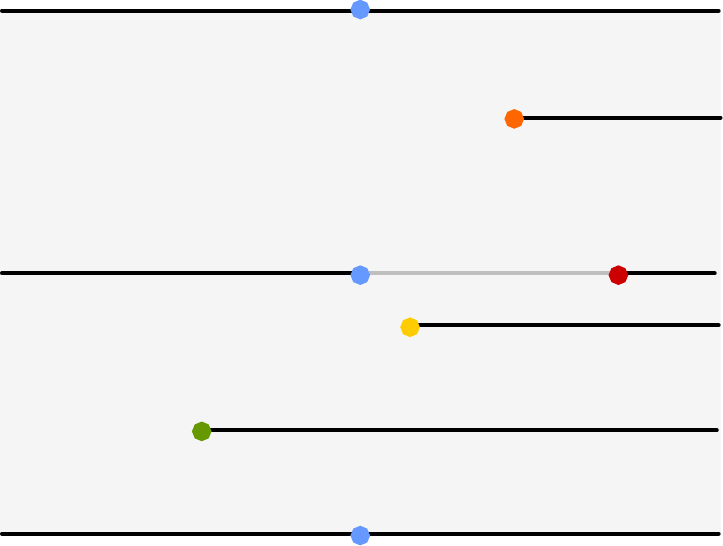}
\caption{Type 2: pushing $v_3$ up from $S_i$ to $S_j$.}
\label{figure:cc2}
\end{figure}

\begin{lemma}
\label{lemma:movesequence}
Let $L, L'$ be connected labeling systems for $\kappa$, and let
\[
df/f \in r(\Omega_L) \cap r(\Omega_{L'})
\]
be a differential (with zeroes and poles labeled) in the image of the strip coordinate domains for both $L$ and $L'$. Then $L'$ can be obtained from $L$ by a sequence of moves of type $1$ and $2$. 
\end{lemma}

\begin{proof}
By hypothesis, there are two translation surfaces $T$ and $T'$ coming from $\Omega_L$ and $\Omega_{L'}$, respectively, that determine the same point in $\Omega_\kappa^{ord}$. Thus $T$ and $T'$ are equivalent via a sequence of cut/paste moves that moreover preserve the labelings of each of the poles and zeroes of the associated differential $df/f$. Applying the strip decomposition to $T$ and $T'$, it follows that each of the corresponding strips are individually cut/paste equivalent. A cut/paste move applied to a given strip corresponds to a move of type $1$ on the labeling system. After applying a cut/paste isomorphism taking $T$ to $T'$ as labeled translation surfaces, the only remaining choices in the assignment of a labeling system arises in assigning prongs lying on the boundary of two strips to one or the other; this corresponds to moves of type $2$. 
\end{proof}

\para{Summary} We summarize the results of the section in the following result, describing the global structure of $\Omega_\kappa^{ord}$ obtained by gluing together strip coordinate patches according to moves of types $1$ and $2$.

\begin{proposition}\label{proposition:cellstruct}
    There is a biholomorphism
    \[
    \Omega_\kappa^{ord} \cong \coprod \Omega_L/ \sim,
    \]
    where the union is taken over all connected labeling systems for $\kappa$ and $\sim$ is the equivalence relation generated by moves of types $1$ and $2$ as in \Cref{lemma:fixedtofree,lemma:pushup}. 
\end{proposition}

\begin{proof}
We claim that the topological space $\coprod \Omega_L/\sim$ is a manifold under the system of coordinates provided by $\Omega_L$. This is not quite immediate from what we have shown - the strip coordinate domains are not open, and their interiors do not quite cover $\Omega_\kappa^{ord}$, missing points where some free prong has period with imaginary part $0$ or $2 \pi i$. But such points lie in the interior of the {\em union} of two strip coordinate domains, e.g. as shown in \Cref{figure:cc2}. By \Cref{lemma:fixedtofree,lemma:pushup}, the transition functions between overlapping $\Omega_L, \Omega_{L'}$ are holomorphic (and indeed affine), and hence $\coprod \Omega_L/\sim$ is a complex manifold.

By \Cref{lemma:stripintoomega} and \Cref{lemma:movesequence}, the set of realization maps $r: \Omega_L \to \Omega_\kappa^{ord}$ assemble into a biholomorphism.
\end{proof}

\begin{remark}
    While one can exhibit a deformation retraction showing that an individual set $r(\Omega_L)$ is contractible, it is not the case that all intersections of sets $r(\Omega_L)$ are contractible. Specifically, if $v$ lies on the top right of its strip, and $v'$ lies on the bottom right of the strip above, then there is a three-fold intersection of labeling systems where at most one of $v$ or $v'$ has been moved up or down. This intersection has two components, arising from the different linear orderings on the real parts. To compute the homotopy type of $\Omega_\kappa^{ord}$ as the nerve of a covering, it is therefore necessary to further subdivide the pieces $\Omega_L$ (taking into account the various orderings of the real parts) so as to account for this phenomenon. 
\end{remark}

\section{Winding numbers}\label{section:wnf}

In this section, we begin our study of the monodromy of strata of polynomials (\Cref{mainthm:monodromy}). Our ultimate objective is \Cref{lemma:monodromycontain}, which asserts that the monodromy image $B_n[\kappa] = \bar \rho(\B_n[\kappa])$ lies in the kernel of a certain crossed homomorphism $\phi_r$. This will be constructed as a measure of ``change of winding number'' for arcs on a translation surface; accordingly, we begin with a discussion of the theory of {\em relative winding number functions}. In \Cref{subsection:braidedGL}, we use the theory of winding number functions to give an example of a braid which satisfies the convexity condition enforced by the Gauss-Lucas theorem, but which nevertheless cannot be realized as the braid of root and critical points of any family of polynomials.

\subsection{Winding number functions} To avoid a lengthy digression, we give here an abbreviated account of the theory of winding number functions which will suffice for our purposes; see \cite[Section 2]{strata3} for a fuller discussion.

\begin{definition}[Relative winding number function]
    Let $\C_{n,p}$ denote the surface $\CP^1$ with three sets of marked points: $n$ points $S_r \subset \C$ which we call the {\em roots}, $p$ points $S_c \subset \C$ called the {\em critical points}, and $\infty$. Let $S = S_r \cup S_c$. We allow for the possibility of tracking only roots, and not critical points, and hence we permit $p = 0$. We further endow $\C_{n,p}$ with a {\em weighting}
    \[
    w: S \to \Z
    \]
    for which $w(z) = -1$ for each root, and $w(z)>0$ for each critical point. In the context under study, we think of $w$ as the function that assigns to each point its order as a zero or a pole.
    
    Let $\mathcal A_{n,p}$ denote the set of isotopy classes of properly-embedded smooth oriented arcs, disjoint from all marked points on their interior, that connect some root in $\C$ to $\infty$ (in that order, relative to the orientation). A {\em $\Z/r\Z$ relative winding number function} is a set map
    \[
    \psi: \mathcal A_{n,p} \to \Z/r\Z
    \]
    that satisfies the {\em twist-linearity condition}
    \begin{equation}\label{equation:twistlin}
    \psi(T_c(a)) = \psi(a) + \pair{a,c} \norm{c},
    \end{equation}
    where $c\subset \C_{n,p}$ is a simple closed curve and $\norm{c}$ is determined by the formula
    \begin{equation}\label{normc}
    \norm{c} = 1 +\sum_{z \in \mbox{int}(c)\cap S}w(z),
    \end{equation}
    where the sum runs over the points of $S$ in the interior of $c$ (i.e. the component of $\CP^1 \setminus c$ not containing $\infty$). As usual, $\pair{a,c}$ denotes the algebraic intersection pairing, relative to the specified orientation on $a$ and the orientation on $c$ for which $\infty$ lies to the left. When $r = 0$, we call such an object an {\em integral relative winding number function}.
\end{definition}

\begin{example}[Horizontal winding number function $\psi_T$]
Let $n \ge 2$ be given, let $\kappa$ be a partition of $n-1$, and let $T \in \Omega_{\kappa}$ be a translation surface structure on $\CP^1$. Such $T$ corresponds to a differential $df/f$, and we let $\C_{n,p}(T)$ be the surface with the roots and critical points of $f$ marked. Let $w$ be the weighting given by the order of the corresponding pole or zero of $df/f$. 

$T$ endows $\C_{n,p}(T)$ with an integral relative winding number function $\psi_T$ called the {\em horizontal winding number function}. Let $a \subset \C_{n,p}(T)$ be a properly-embedded smooth oriented arc connecting a pole of $df/f$  to $\infty$. We assign the value $\psi_T(a) \in \Z$ as follows: realize $a$ as an arc on the translation surface $T$ not passing through any of the cone points. As $a$ connects a zero of $f$ to the pole at $\infty$, it runs from left to right on $T$, and as it is properly embedded, it can be isotoped so that it follows a leaf of the horizontal foliation outside of some compact region of $T$. Such a representative carries an integral winding number $wn_T(a) \in \Z$ by measuring the winding of the forward-pointing tangent vector relative to the horizontal vector field (the winding number is {\em integral} because of the condition that the arc coincide with a leaf of the horizontal foliation outside of a compact region).  

\begin{lemma}
The function
\[
\psi_T(a) := wn_T(a)
\]
is a well-defined integral relative winding number function on $\C_{n,p}(T)$. 

Setting $r = \gcd(\kappa)$, the mod-$r$ reduction
\[
\bar{\psi}_T(a) := \psi_T(\tilde a) \pmod r,
\]
where $a \in \C_n(T)$ is an arc and $\tilde a \in \C_{n,p}(T)$ is an arbitrary lift, 
is a well-defined $\Z/r\Z$ relative winding number function on $\C_n(T)$.
\end{lemma}
\begin{proof}
    To see that $\psi_T$ is well-defined, we must check (i) that $wn_T(a)$ is unchanged by an isotopy of $a$, and (ii) that $\psi_T(a)$ satisfies the twist-linearity condition \eqref{equation:twistlin}. To see that $\bar{\psi}_T(a)$ is well-defined, we must further check (iii) that $wn_T(\tilde a)$ is unchanged mod $r$ by an isotopy of $\tilde a$ across a cone point of $T$.

    To establish (i), we recall that $a$ is horizontal except on a compact set. As the winding number of such an arc is integral (and hence discretely-valued), it follows that the winding number is invariant under any compactly-supported isotopy. Under an isotopy with noncompact support, $a$ can wrap around a pole some number of times, potentially altering the winding number. But this is in fact a special case of (ii): winding $a$ around a pole of $T$ is equivalent to applying the Dehn twist around a curve $c$ enclosing this single pole, i.e. for which $\norm{c} = 0$.
   
   To establish (ii), let $c \subset \C_{n,p}$ be a simple closed curve. It follows from the Poincar\'e-Hopf theorem that the winding number of $c$ on $T$ is given by $\norm{c}$ as in \eqref{normc}, as the index of the horizontal vector field at a root or critical point is given by $w$. Applying the Dehn twist $T_c$ to $a$, we see that
   \[
    wn_T(T_c(a)) = wn_T(a) + \pair{a,c} \norm{c}
   \]
   holds, since at each intersection between $a$ and $c$, the twist $T_c(a)$ wraps once around $c$, contributing $\pm \norm{c}$ to the winding number, the sign determined by the sign of the intersection. 

    For (iii), we again invoke the Poincar\'e-Hopf theorem to see that as a curve is isotoped across a zero of index $k$ on a vector field, the winding number changes by $k$. By hypothesis, the order of each zero is divisible by $r$. Thus, after reducing mod $r$, the quantity $wn_T(\tilde a)$ is independent of the choice of lift $\tilde a$ of $a \subset \C_n(T)$ to $\C_{n,p}(T)$. 
\end{proof}
\end{example}

As $B_{n,p}$ acts on the set $\mathcal A_{n,p}$ of arcs, there is an induced action
    \[
    \beta \cdot \psi( a) = \psi(\beta^{-1}a)
    \]
    on the set of relative winding number functions, and hence there is an associated stabilizer subgroup of $B_{n,p}$, which we call the {\em framed braid group}. In the case where we track roots but not critical points, we call such groups {\em $r$-spin braid groups}, by analogy with the theory of ``$r$-spin structures'' and their associated ``$r$-spin mapping class groups'' in higher genus, cf. \cite{strata2}.

\begin{definition}[Framed braid group $B_{n,p}(\psi)$, $r$-spin braid group $B_n(\bar \psi)$]
    Let $\psi$ be an integral relative winding number function on $\C_{n,p}$. The associated {\em framed braid group} $B_{n,p}(\psi)$ is the subgroup of $B_{n,p}$ stabilizing $\psi$ under the above action on the set of integral relative winding number functions.

    Likewise, if $\bar \psi$ is a $\Z/r\Z$ relative winding number function on $\C_n$, the associated {\em $r$-spin braid group} $B_n(\bar\psi)$ is the stabilizer of $\bar \psi$.
\end{definition}

In \Cref{lemma:rspinbraid}, we will see that the monodromy of a stratum $\Conf_n(\C)[\kappa]$ is contained in a certain framed braid group. To establish this, we must digress briefly to give a precise construction of the monodromy homomorphism.

\begin{definition}[Monodromy]\label{definition:monodromy}
Let $\kappa$ be a partition of $n-1$ with $\abs{\kappa}= p$ parts, and let $B_{n,p}$ be the subgroup of $B_{n+p}$ preserving the division of the $n+p$ strands into groups of size $n$ and $p$. Recalling the definition $\B_n[\kappa]:= \pi_1(\Conf_n(\C)[\kappa])$, the {\em monodromy} is a homomorphism
\[
\rho: \B_n[\kappa] \to B_{n,p}
\]
constructed as follows. Let $f \in \Conf_n(\C)[\kappa]$ be chosen as a basepoint, and let $T = df/f$ be the associated translation surface in $\Omega_\kappa$. Fix a choice of {\em marking} (i.e. homeomorphism) $\mu: \C_{n,p} \to T$. Let $\beta: [0,1] \to \Conf_n(\C)[\kappa]$ be a loop based at $f$, which induces a loop in $\Omega_\kappa$, which we will also write $\beta$; we write the image of this latter loop as $\beta(t) = T_t$ with $T_0 = T_1 = T$. The family $\{T_t\}$ of translation surfaces over $[0,1]$ is topologically trivial, and hence there is a well-defined isotopy class of identification $f_t: T_0 \to T_t$ for $t \in [0,1]$, which induces a {\em propagation} $\mu_t: \C_{n,p} \to T_t$ of the marking map. The {\em monodromy} of $\beta$ is the element 
\[
\rho(\beta) = \mu_1^{-1} \mu_0 \in \Mod(\C_{n,p}),
\]
where $\Mod(\C_{n,p})$ denotes the mapping class group of $\C_{n,p}$. As the marking $\mu$ can be enhanced to identify a tangent vector at $\infty \in \C_{n,p}$ with the canonical horizontal direction on translation surfaces in $\Omega_\kappa$, we can identify $\Mod(\C_{n,p})$ with the mapping class group of the $n,p$-times punctured disk $D_{n,p}$, i.e. the subgroup $B_{n,p} \le B_{n+p}$ preserving setwise the roots and critical points. 

That $\rho: \B_n[\kappa] \to B_{n,p}$ is a homomorphism is a consequence of the fact that if $\beta, \gamma$ are loops for which the propagated markings at $t = 1$ are denoted $\mu_\beta, \mu_\gamma: \C_{n,p} \to T$, then $\mu_\gamma \mu_0^{-1} \mu_\beta$ gives a propagation of the marking along the composite path $\beta \gamma$. 

Note that $\rho$ is not completely canonical: it depends on a choice of marking $\mu_0: \C_{n,p} \to T$ (and in particular depends on a choice of basepoint $T \in \Omega_\kappa$). However, it is easy to see that different choices of marking lead to {\em conjugate} monodromy homomorphisms.

Note also that we obtain a reduction 
\[
\bar \rho: \B_n[\kappa] \to B_n
\]
by forgetting the braid of the critical points.
\end{definition}

Define
\[
B_{n,p}[\kappa] := \rho(\B_n[\kappa])
\]
and
\[
B_n[\kappa] := \bar \rho(\B_n[\kappa]).
\]

\begin{lemma}\label{lemma:rspinbraid}
    Let $T \in \Omega_\kappa$ be a basepoint. Under the monodromy map $\rho: \B_n[\kappa] \to B_{n,p}$ based at $T$, there are containments
    \[
    B_{n,p}[\kappa] \le B_{n,p}(\psi_T)
    \]
   and 
    \[
    B_n[\kappa] \le B_n(\bar \psi_T).
    \]
\end{lemma}
\begin{proof}
    We assume the notation of \Cref{definition:monodromy}, and consider the monodromy of a loop $\beta$ in $\Conf_n(\C)[\kappa]$. As the marking is propagated along the loop, this induces an identification of the sets $\mathcal A_{n,p}(T_t)$ of the arcs on the translation surfaces $T_t$. It follows that $\beta$ induces a continuously-varying family of winding number functions $\psi_{T_t}$ on the set $\mathcal A_{n,p}$ of arcs on the reference surface $\C_{n,p}$. As the set of winding number functions is a discrete set, it follows that all such winding number functions coincide. In particular, $\psi_{T_0} = \psi_{T_1}$, but from the definitions we have $\psi_{T_1} = \beta \cdot \psi_{T_0}$, showing that $B_{n,p}[\kappa] \le B_{n,p}(\psi_T)$ as claimed. 

    The containment $B_n[\kappa] \le B_n(\bar \psi_T)$ is a straightforward consequence of the fact that the integral relative winding number function $\psi_T$ on $\C_{n,p}(T)$ descends to the $\Z/r\Z$ relative winding number function $\bar \psi_T$ on $\C_n(T)$ under the forgetful map $\C_{n,p}(T) \to C_n(T)$.  
\end{proof}

\subsection{Convexity is not enough: the braided Gauss-Lucas theorem}\label{subsection:braidedGL}
To illustrate \Cref{lemma:rspinbraid}, we give here in \Cref{figure:nonreal} an example of a braid in $B_{n,p}$ that admits a ``convex representative'', i.e. where the $p$-stranded braid of critical points lies inside the convex hull of the $n$-stranded braid of roots for all times $t$, and yet which does not arise from any loop of polynomials.

\begin{figure}[ht]
\labellist
\small
\endlabellist
\includegraphics[scale=0.7]{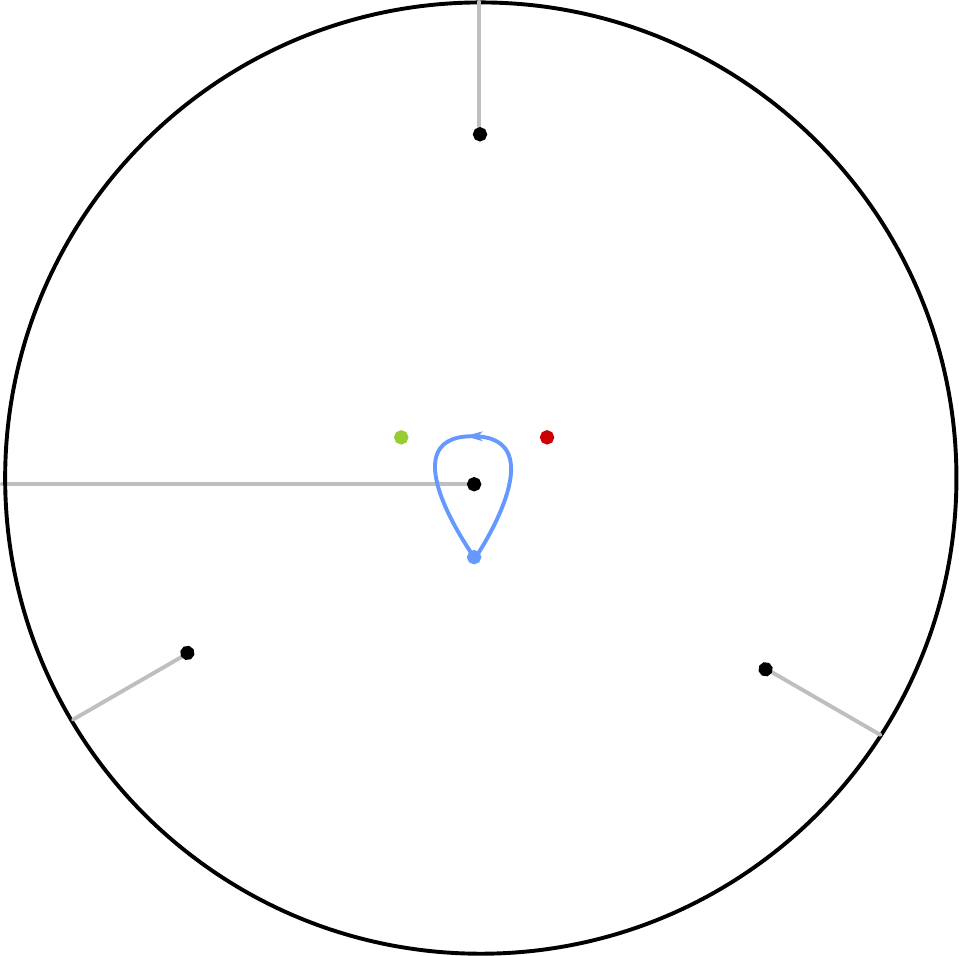}
\caption{A braid which cannot be realized by a family of polynomials in $\Conf_4(\C)[1^3]$. The four roots are illustrated in black, and the three simple critical points are colored. A choice of arcs connecting roots to infinity (depicted here as the entire boundary for visual simplicity) are shown in gray. As the blue point orbits the central root, it alters the winding number of the corresponding arc (as can be seen from the twist-linearity formula), and is thus not contained in the framed braid group $B_{4,3}(\psi_T)$. By \Cref{lemma:rspinbraid}, it follows that this braid cannot be realized by a loop in $\Conf_4(\C)[1^3]$.}
\label{figure:nonreal}
\end{figure}

\subsection{Mod-$r$ winding numbers as crossed homomorphisms}
From here to the end of the paper, we will concentrate on the monodromy $\bar \rho$ of the roots only, leaving a study of the refinement $\rho$ for future work.

Here, we show that the $r$-spin braid group $B_n(\bar \psi_T)$ can be identified with the kernel of a certain crossed homomorphism $\phi_\kappa$, and show that $\phi_\kappa$ has a very simple formula; as this ultimately depends only on $r = \gcd(\kappa)$ and not $\kappa$ itself, in the sequel we will work instead with the equivalent crossed homomorphism $\phi_r$ with the simple formula.

\begin{lemma}\label{lemma:monodromycontain}
    Let $T \in \Omega_\kappa$, and let $S_1, \dots, S_n$ be the strips in a strip decomposition for $T$. For $i = 1, \dots, n$, let $a_i \subset \C_n$ be an arc corresponding to a horizontal leaf on $T$ contained entirely in $S_i$. Then the function
    \begin{align*}
    \phi_\kappa: B_n &\to (\Z/r\Z)^n \\
    \beta &\mapsto \sum \bar\psi_T(\beta^{-1} a_i)e_i 
    \end{align*}
    is equal to the crossed homomorphism 
    \begin{align*}
    \phi_r: B_n &\to (\Z/r\Z)^n \\
    \sigma_i &\mapsto e_{i+1}
    \end{align*}
    (where $B_n$ acts on $(\Z/r\Z)^n$ on the left via the coordinate-permutation action induced from the quotient $B_n \to S_n,\ \beta \mapsto \bar \beta$). 

    In particular, there is a containment
    \[
    B_n[\kappa] \le \ker(\phi_\kappa) = \ker(\phi_r).
    \]
\end{lemma}
\begin{proof}
   To establish that $\phi_\kappa$ is a crossed homomorphism, we make the following observation. If $a$ and $a'$ are two arcs on $T$ with the same beginning and end points, then $a \cup a'$ is an oriented closed curve. There are two cusps at the common endpoints, and otherwise $a \cup a'$ is smoothly immersed. By the Poincar\'e-Hopf theorem, the winding number of $a \cup a'$ (reduced mod $r$, as usual) counts the total number of poles on $T$ enclosed by $a \cup a'$ (up to a correction factor of $1$ coming from the change in winding number arising from smoothing out the cusps). Thus this quantity is invariant under the action of the braid group:
    \[
    \bar \psi_T(\beta a) - \bar \psi_T(\beta a') = \bar\psi_T(a) - \bar\psi_T(a').
    \] 

   Now given $\alpha, \beta \in B_n$, we use this to compute
    \begin{align*}
    \phi_\kappa(\alpha \beta) &= \sum \bar\psi_T(\beta^{-1} \alpha^{-1} a_i)e_i\\
    &= \sum \left(\bar\psi_T(\beta^{-1} \alpha^{-1} a_i)- \bar\psi_T(\beta^{-1} a_{\bar{\alpha^{-1}}i}) + \bar\psi_T(\beta^{-1} a_{\bar{\alpha^{-1}}i})\right) e_i \\
     &= \sum \left(\bar\psi_T(\alpha^{-1}a_i) - \bar\psi_T(a_{\bar{\alpha^{-1}}i}) + \bar\psi_T(\beta^{-1} a_{\bar{\alpha^{-1}}i})\right) e_i.    
    \end{align*}
    Splitting into three vectors, we observe that the first is $\phi_\kappa(\alpha)$, the second is identically zero (each component $\bar\psi_T(a_i)$ is zero since $a_i$ is a leaf of the horizontal foliation), and the third is identified as $\alpha \cdot \phi_\kappa(\beta)$. Thus $\phi_\kappa$ is a crossed homomorphism as claimed.

 To identify $\phi_\kappa$ with $\phi_r$, it suffices to check equality on the standard generators $\sigma_i$. Under the {\em standard marking} shown in \Cref{figure:refsurface} below, we see that $\sigma_i^{-1}$ takes $a_i$ to $a_{i+1}$ and $a_{i+1}$ to $T_{c_i}a_i$, where $c_i$ is the boundary of the {\em standard arc} $A_{i, i+1}$ connecting marked points $i$ and $i+1$ (cf. \Cref{definition:standardarc} below). By the twist-linearity formula, it follows that
 \[
 \psi_T(\sigma_i^{-1} a_j) = \begin{cases}
     0 & j \ne i+1\\ 1 & j = i+1
 \end{cases}
 \]
 from which the claim follows.
\end{proof}

\section{Constructing monodromy elements}\label{section:construct}

In this section, we ``fill out'' the monodromy image of $\Conf_n(\C)[\kappa]$, showing that the image $B_n[\kappa]$ contains the subgroup $\Gamma_n^r$ of ``basic $(r+1)^{st}$ twists''. This group is defined in \Cref{definition:rthtwist} below; we exhibit some monodromy elements in \Cref{lemma:monodromygens}, and after some group theory carried out in \Cref{lemma:gcd}, we show the containment $\Gamma_n^r \le B_n[\kappa]$ in \Cref{lemma:rthroots}.

\begin{definition}[Basic $(r+1)^{st}$ twist $\Sigma_{a;r}$, subgroup $\Gamma_n^r$]\label{definition:rthtwist}
For $n\ge a+r$, the {\em basic $(r+1)^{st}$ twist} $\Sigma_{a;r} \in B_n$ is defined to be the element
\[
\Sigma_{a;r} = \sigma_a \dots \sigma_{a + r - 1}.
\]
The {\em basic $(r+1)^{st}$ twist group} is the subgroup
\[
\Gamma_n^r \le B_n
\]
generated by the set of basic $(r+1)^{st}$ twists.
\end{definition}

\begin{remark}\label{remark:undercrossing}
    Pictorially, the basic twist $\Sigma_{a;r}$ is given by taking the strand in position $a$ and crossing it over the next $r$ strands to the right, and the inverse $\Sigma_{a;r}^{-1}$ is the same but with the strand crossing over $r$ strands to the left. In particular, {\em $\beta \in \Gamma_n^r$ if and only if it admits a diagram for which each overcrossing passes over a multiple of $r$ strands below it.}
\end{remark}

\begin{lemma}
    \label{lemma:monodromygens}
    Let
    \[
    \kappa = \{k_1, \dots, k_p\}
    \]
    be a partition of $n-1$. Then $B_n[\kappa]$ contains the elements 
    \[
    \Sigma_{1;k_1},\quad \Sigma_{k_1+1,k_2},\quad \dots,\quad  \Sigma_{k_1 + \dots + k_{p-1}+1;k_p}.
    \]
\end{lemma}
\begin{proof}
    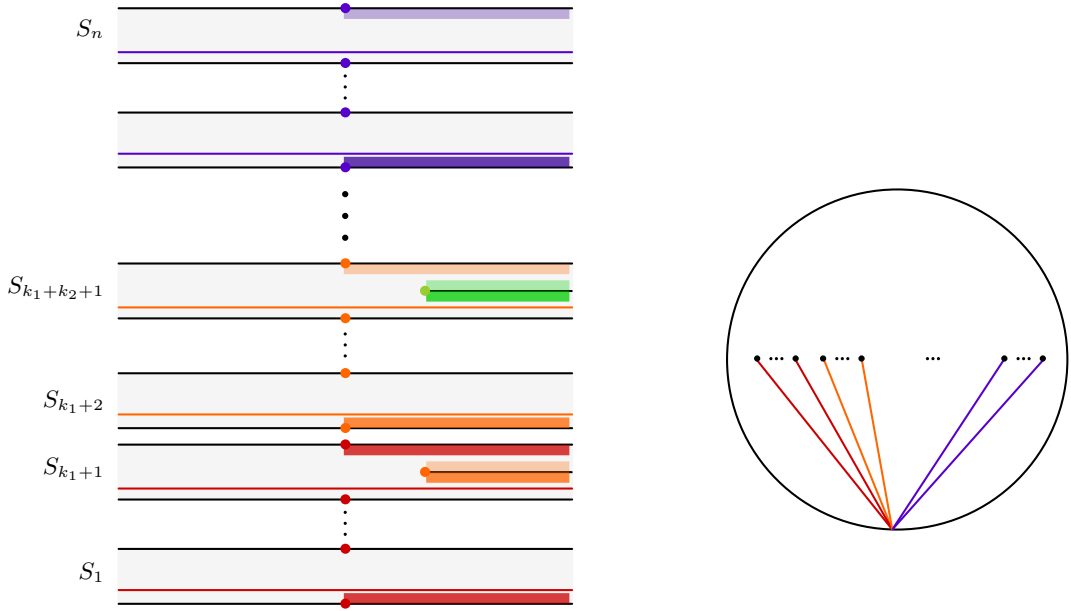
\begin{figure}[ht]
\labellist
\small
\pinlabel $S_1$ [r] at -2.83 19.84
\pinlabel $S_{k_1+1}$ [r] at -2.83 76.53
\pinlabel $S_{k_1+2}$ [r] at -2.83 113.38
\pinlabel $S_{k_1+k_2+1}$ [r] at -2.83 175.73
\pinlabel $S_n$ [r] at -2.83 317.45
\endlabellist
\includegraphics[scale=0.7]{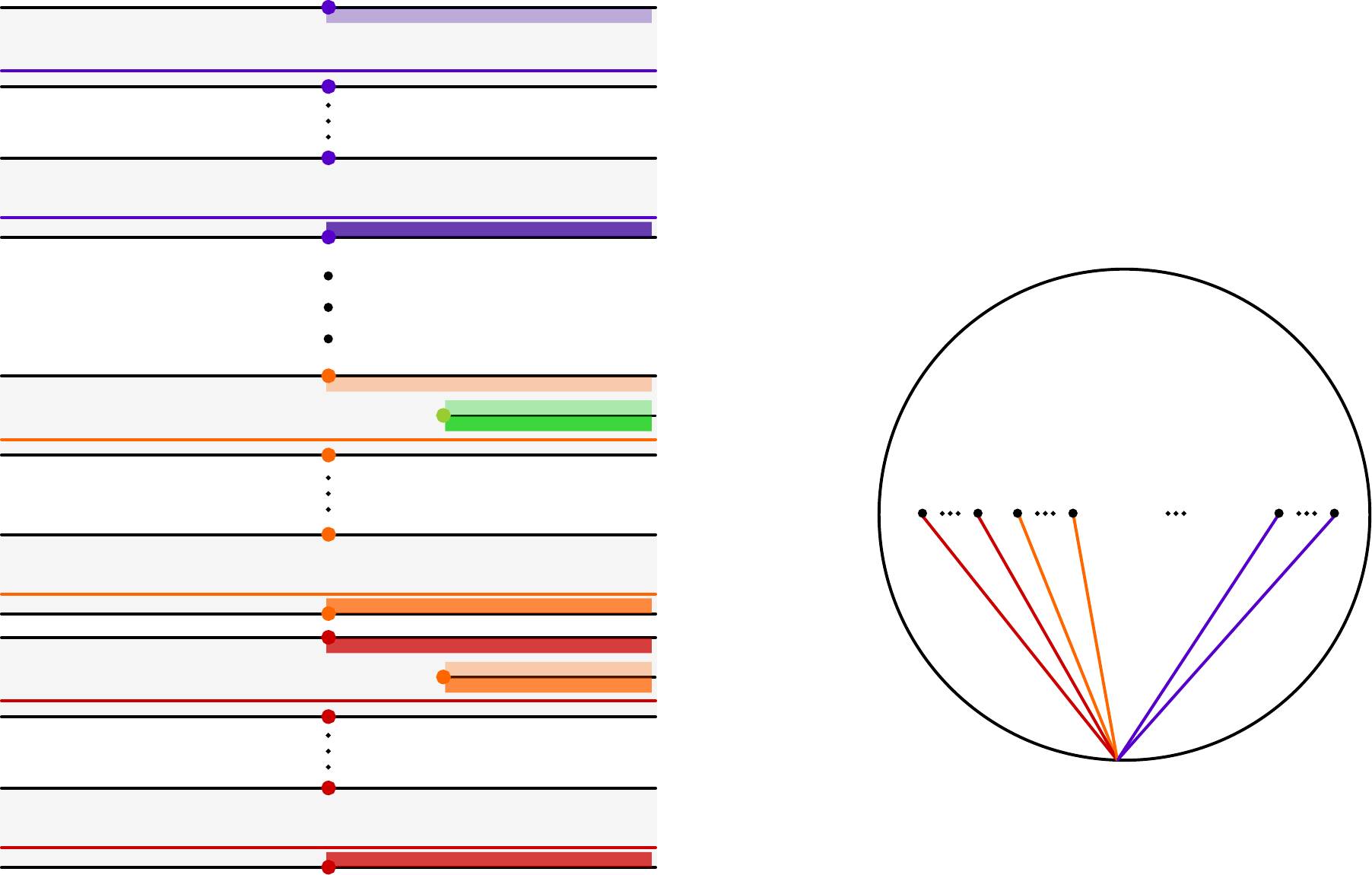}
\caption{The standard marking. At left, the reference translation surface $T_\kappa$ for the stratum $\kappa = \{k_1, \dots, k_p\}$. The top left and bottom left of each strip are identified. There are $p$ blocks of strips, each one corresponding to a given cone point (depicted as the colored points). Within each block of strips of the same color, the top right segment of $S_i$ is identified to the bottom right on $S_{i+1}$, with remaining gluing instructions specified by color as in \Cref{figure:sd1}. All but the bottom-most cone point (red in the figure) have one free prong in the top strip of the block below. The colored horizontal lines equip $T_\kappa$ with a marking. At right, the corresponding standard marking of the $n$-punctured disk.}
\label{figure:refsurface}
\end{figure}

Consider the ``standard marking'' of the translation surface $T_\kappa \in \Omega_\kappa$ shown in \Cref{figure:refsurface}. In \Cref{figure:monodromyloop}, we exhibit loops in $\Omega_\kappa$ based at $T_\kappa$. By comparing markings of the surface before and after, we compute their monodromy in $B_n$ to be $\Sigma_{k_1 + \dots + k_{i-1} + 1; k_i}^{-1}$. Recalling from \eqref{equation:les} that the projection $\B_n[\kappa] \to \pi_1^{ord}(\Omega_\kappa)$ is surjective, we see that we can lift these loops to $\Conf_n(\C)[\kappa]$, realizing them as elements of the monodromy group $B_n[\kappa]$.
\end{proof}
    \begin{figure}[ht]
\labellist
\small
\pinlabel recut [b] at 277.10 544.20
\pinlabel push [br] at 282.77 422.32
\pinlabel reorder [b] at 279.93 320.28
\pinlabel recut [br] at 279.93 195.57
\pinlabel change~of~marking at 435 120
\endlabellist
\includegraphics[scale=0.7]{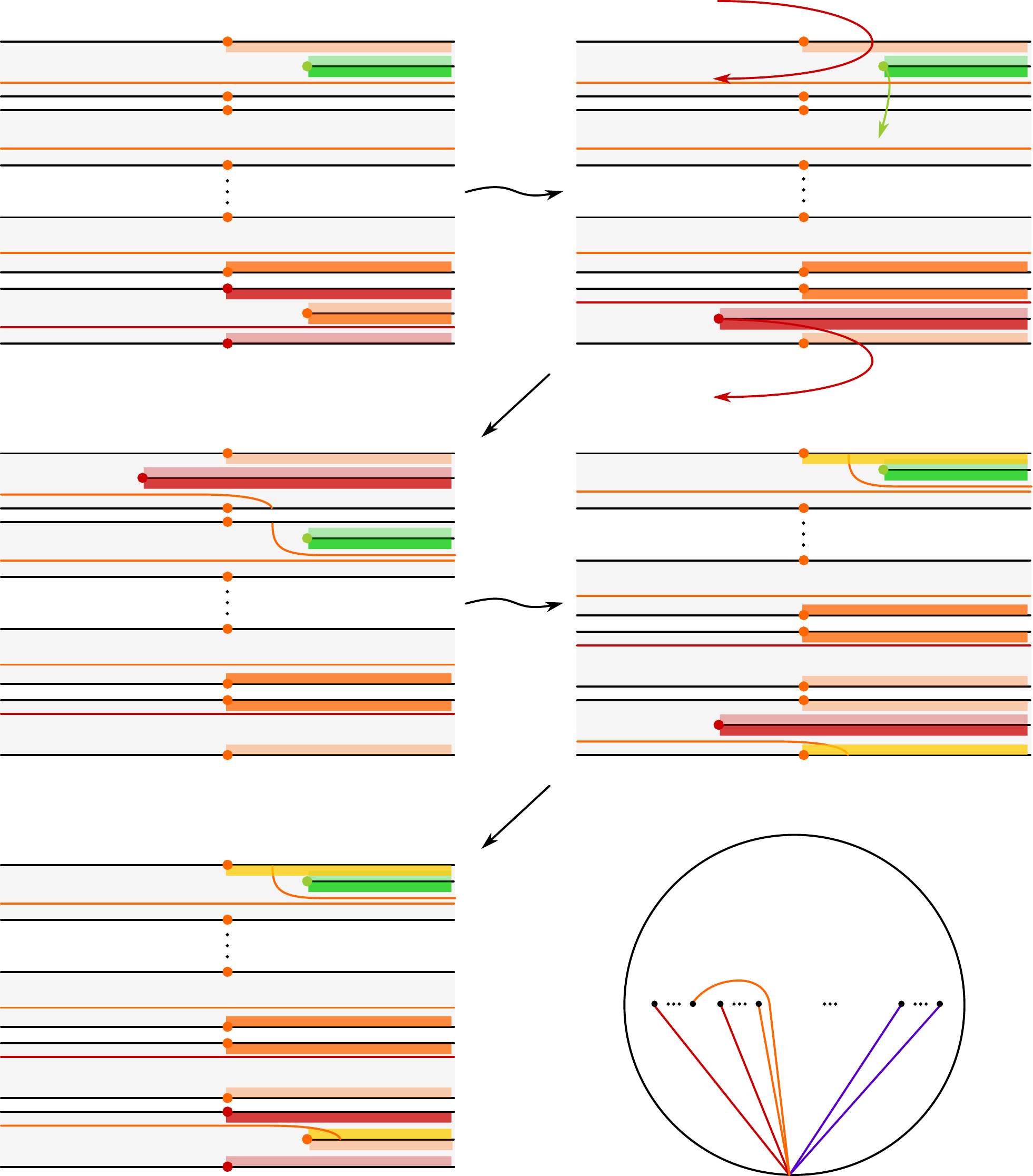}
\caption{We depict the block of strips on $T_\kappa$ between $S_{k_1 + \dots + k_{i-1} + 1}$ and $S_{k_1 + \dots + k_{i}}$. Reading lexicographically, we first recut the bottom strip $S_{k_1 + \dots + k_{i-1} + 1}$ so that it is bounded by the same cone point as the rest. Then we push each of the free prongs down one strip. Next, we apply a cut/paste move to reorder the strips, moving each one up one spot. Finally, we recut the bottom strip once again so that it is bounded by the other cone point. Note that in the case of the bottom block, there is no free prong in the bottom strip, in which case we skip the recutting steps, and in the case of the top block, there is no free prong in the top strip, and there is a slightly different picture (omitted). The picture at bottom right depicts the change of marking, i.e. the monodromy of the loop.}
\label{figure:monodromyloop}
\end{figure}
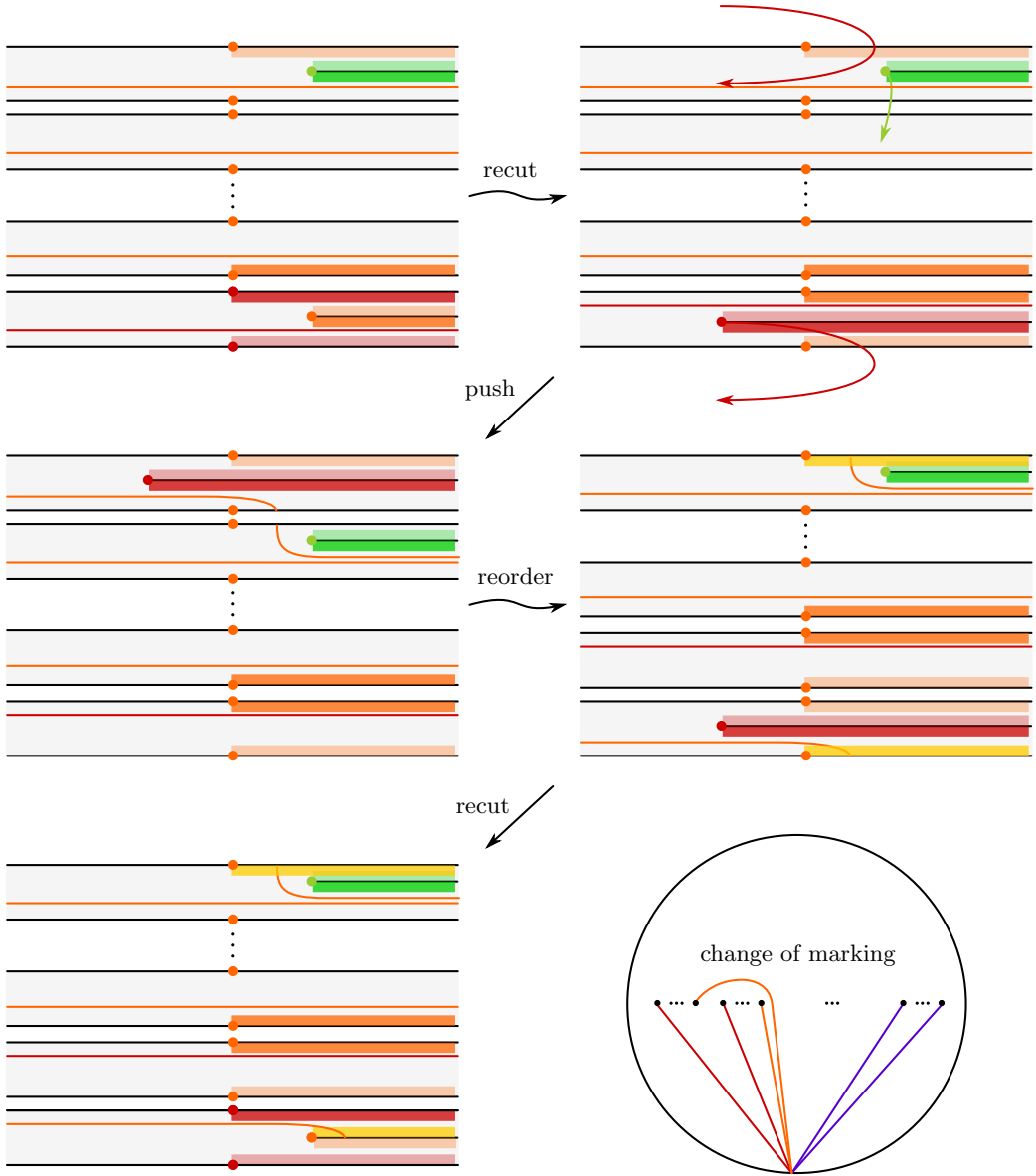

\begin{lemma}
    \label{lemma:gcd}
    Let $a,b$ be integers, and let $G \le B_{a+b+1}$ be the subgroup generated by $\Sigma_{1;a}$ and $\Sigma_{a+1,b}$. Then $\Sigma_{1;\gcd(a,b)} \in G$.
\end{lemma}
\begin{proof}
    Observe that $\Sigma_{1;a}\Sigma_{a+1;b} = \Sigma_{1;a+b}$, and that
    \begin{equation}\label{equation:conjformula}
    \Sigma_{1;a+b}^j\ \Sigma_{p;q}\ \Sigma_{1;a+b}^{-j} = \Sigma_{p+j;q}
    \end{equation}
    so long as the indices $p+j, \dots, p+j + q -1$ lie on the interval $[1,a+b]$. Thus by conjugating, we can shift the first index of any $\Sigma_{p;q}$ to any valid position, and by taking $\Sigma_{p;q}^{-1} \Sigma_{p;r}$ for $q > r$ and conjugating, we obtain $\Sigma_{p;r-q}$ from $\Sigma_{p;q}$ and $\Sigma_{p;r}$. By repeatedly shifting and deleting initial segments in this way, we can perform the Euclidean algorithm on $a,b$, eventually obtaining $\Sigma_{1;\gcd(a,b)}$.
\end{proof}

\begin{lemma}
    \label{lemma:rthroots}
    For any $n \ge 2$ and any partition $\kappa$ of $n-1$, the group $B_n[\kappa]$ contains the elements $\Sigma_{1;r}$ and $\Sigma_{1;n-1}$, and hence every basic $(r+1)^{st}$ twist $\Sigma_{k;r}$. Thus,
    \[
    \Gamma_n^r \le B_n[\kappa].
    \]
\end{lemma}
\begin{proof}
    By \Cref{lemma:monodromygens}, $B_n[\kappa]$ contains the elements
    \[
    \Sigma_{1;k_1},\quad \Sigma_{k_1+1,k_2},\quad \dots,\quad  \Sigma_{k_1 + \dots + k_{p-1}-1;k_p}.
    \]
    By repeated application of \Cref{lemma:gcd}, one sees that $\Sigma_{1;r} \in B_n[\kappa]$, and also
    \[
    \Sigma_{1;n-1} = \Sigma_{1;k_1}\ \Sigma_{k_1+1,k_2}\ \dots \  \Sigma_{k_1 + \dots + k_{p-1}-1;k_p} \in B_n[\kappa].
    \]
    By \eqref{equation:conjformula}, $B_n[\kappa]$ thus contains all $\Sigma_{k;r}$.
\end{proof}

\begin{corollary}
    Let $n \ge 2$ be given, and let $\kappa$ be a partition of $n-1$ for which $r = \gcd(\kappa) = 1$. Then the monodromy map $\bar \rho: \B_n[\kappa] \to B_n$ is surjective.
\end{corollary}
\begin{proof}
    By \Cref{lemma:rthroots}, the image of $\bar\rho$ contains all basic $(r+1)^{st}$ twists $\Sigma_{k;r}$, but for $r = 1$ these are just the standard half-twist generators of $B_n$.
\end{proof}



\section{Generating $\ker(\phi_r)$}\label{section:generating}

In the previous two sections, we have seen how the monodromy image $B_n[\kappa]$ is contained in the kernel of a crossed homomorphism $\phi_r$, and conversely {\em contains} the subgroup $\Gamma_n^r$ of basic $(r+1)^{st}$ twists. Here, we complete the circle of containments, showing that when $n$ is sufficiently large compared to $r$, the kernel of $\phi_r$ is generated by basic $(r+1)^{st}$ twists. 

 We must first specify what is meant by ``sufficiently large''. Define
    \begin{equation}\label{equation:nord}
    n_0(r,d) = \begin{cases}
            8   & r = 2\\
            (6+d)r & r \mbox{ odd}\\
            (12+d)r & r\ge 4 \mbox{ even}.
    \end{cases}
    \end{equation}

\begin{theorem}
    \label{theorem:genkernel}
    Let $n \ge 3$ and $r \ge 2$ be given; let $d \in \{0,1,2\}$ be the remainder of $n/3$. Then for $n \ge n_0(r,d)$, the kernel of $\phi_r$ is generated by $\sigma_1 \dots \sigma_r$ and $\sigma_1 \dots \sigma_{n-1}$. 
\end{theorem}

The material of this section is purely braid-theoretic and does not require any knowledge e.g. of winding number functions. The outline is as follows. In \Cref{subsection:virtual}, we discuss a new crossed homomorphism $\Upsilon_r$, which can be computed graphically given a braid diagram as a count of ``virtual undercrossings''; we show in \Cref{lemma:equivalent} that $\Upsilon_r = \phi_r$. In \Cref{subsection:factorize}, we use this graphical reformulation to give an algorithm for factoring an element of $\ker(\phi_r)$ supported on a small number of strands into the group $\Gamma_n^r$ of basic $(r+1)^{st}$ twists. Finally in \Cref{subsection:endgame}, we exploit the factorization algorithm to show the equality $\ker(\phi_r) = \Gamma_n^r$, first in \Cref{lemma:phionpb} on the level of the pure braid group, and finally in \Cref{theorem:gammakernel} in general. 

\subsection{$\phi_r$ as a count of virtual undercrossings}\label{subsection:virtual}

\begin{definition}[Virtual undercrossing map]
    \label{definition:undercrossing}
    Let $r \ge 1$ be given. The {\em virtual undercrossing map} is the homomorphism\footnote{That this is indeed a homomorphism is verified by a routine calculation.} $\Upsilon_r: B_n \to \GL_{n+1}(\Z/r\Z) \ltimes (\Z/r\Z)^{n+1}$ defined as follows. Number the components of $(\Z/r\Z)^{n+1}$ from $0$ to $n$, and, for $1 \le i \le n-1$, let $P_i \in \GL_{n+1}(\Z/r \Z)$ be the matrix obtained from $I_{n+1}$ by replacing the $i^{th}$ column with $e_{i-1} - e_i + e_{i+1}$. Then define
    \[
    \Upsilon_r(\sigma_i) = (P_i, e_{i+1}-e_i)
    \]
    for $1 \le i \le n-1$. For $\beta \in B_n$, write
    \[
    \Upsilon_r(\beta) = (M(\beta), v(\beta)). 
    \]
    As is common to all homomorphisms into semi-direct products, the second factor $v(\beta)$ defines a crossed homomorphism $v: B_n \to (\Z/r\Z)^{n+1}$ under the action of $B_n$ on $(\Z/r\Z)^{n+1}$ via $M$.

    Also note that $\Upsilon_r$ defines an action of $B_n$ on $(\Z/r\Z)^{n+1}$ via 
\begin{equation}\label{equation:beta}
\beta\cdot \vec x = M(\beta)\vec x + v(\beta).
\end{equation}
\end{definition}

\begin{lemma}
    \label{lemma:equivalent}
    Let $f: (\Z/r\Z)^n \to (\Z/r\Z)^{n+1}$ be given by 
    \[
    e_i \in (\Z/r\Z)^n \mapsto e_i - e_{i-1} \in (\Z/r\Z)^{n+1}.
    \]
    Then $f$ induces a map of $B_n$-modules, where $(\Z/r\Z)^n$ carries the standard permutation action of $S_n$ and $(\Z/r\Z)^{n+1}$ carries the action via $M$. Under the induced map on homology,
    \[
    f_*(\phi_r) = v.
    \]
    Moreover, $\ker(\phi_r) = \ker(v)$.
\end{lemma}
\begin{proof}
That $f$ is a map of $B_n$-modules under the indicated actions is a routine calculation. To see that $f_*(\phi_r) = v$, it suffices to verify this on the standard generators $\sigma_i$. To that end, we compute
\[
f_*(\phi_r)(\sigma_i) = f(\phi_r(\sigma_i)) = f(e_{i+1}) = e_{i+1} - e_{i} = v(r).
\]
As $v = f_*(\phi_r)$, there is a containment $\ker(\phi_r) \le \ker(v)$, and as $f$ is readily seen to be an injection, it follows that this containment is an equality. 
\end{proof}

\para{Virtual undercrossings} There is a graphical description of $\Upsilon_r$ which provides the key tool for expressing the kernel of these crossed homomorphisms in terms of $(r+1)^{st}$-twists. Let $\beta \in B_n$ be given. We imagine $\beta$ (depicted in black) as sitting ``on top of'' a trivial braid (in blue) with a very large number of strands, where the ends of the blue strands are not fully fixed but are allowed to ``slide'' horizontally. Given $\vec x \in (\Z/r\Z)^{n+1}$, we interpret the entries $x_0, \dots, x_n$ as a mod-$r$ count of the number of strands in the bottom (blue) layer positioned in between each pair of adjacent strands of $\beta$ at the top of the figure. To compute the action of $\beta$ on $\vec x$ via $\Upsilon_r$, we thread the strands in the bottom layer downwards, subject to the rule that strands in the bottom layer never cross, and that at each crossing of $\beta$, the total number of strands crossing under (counting both layers) is $0$ mod $r$. 

\begin{figure}[ht]
\labellist
\tiny
\pinlabel $\red{x_{i-1}}$ [b] at 19.84 96.20
\pinlabel $\red{x_{i}}$ [b] at 65.19 96.20
\pinlabel \red{$x_{i+1}$} [b] at 85.03 96.20
\pinlabel \red{$x_{i-1}+x_i$} [t] at 10 0.00
\pinlabel \red{$-x_i-1$} [t] at 55 0.00
\pinlabel \red{$x_i+x_{i+1}+1$} [t] at 105 0.00
\endlabellist
\includegraphics[scale=1]{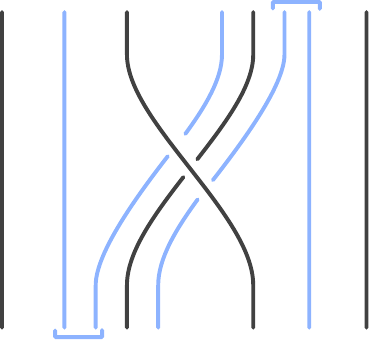}
\caption{The virtual undercrossing procedure for the braid $\sigma_i$. The blue strands are imagined as lying in a layer below the black strands of $\sigma_i$. Each blue strand is actually composed of a very large number of individual strands; the numbers of such strands mod $r$ are depicted above and below. By convention, all strands in the middle group (labeled $x_i$ at the top) must cross in the same direction as the undercrossing in $\beta$. We then split the right strand (labeled $x_{i+1}$ at the top) so that altogether, the number of strands crossing under (both blue and black) is $0$ mod $r$. As all $x_i$ blue strands in the middle group must cross under, and $\sigma_i$ itself contributes one, we must borrow $-x_i - 1 \pmod r$ from the right strand to satisfy this condition.} 
\label{figure:undercrossing}
\end{figure}
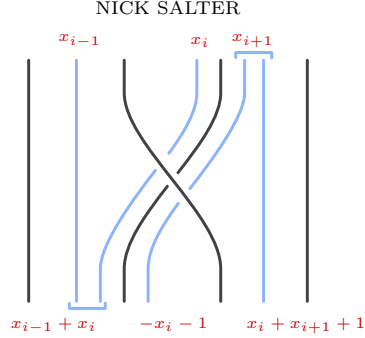

\Cref{figure:undercrossing} illustrates this procedure in the case of a single crossing $\sigma_i$, and shows that the effect on the vector $\vec x = (x_0, \dots, x_n)$ is exactly given by $\sigma_i \cdot \vec x$ as in \eqref{equation:beta}. To compute this action for a general braid, we simply repeat this process at each crossing, working from top to bottom. In particular, the value $v(\beta)$ is computed as the output of the virtual undercrossing procedure applying the zero vector at the top of the braid diagram for $\beta$. For future reference, we record the following characterization of $\ker(\phi_r)$. 

\begin{lemma}
    \label{lemma:kernel0}
    A braid $\beta \in B_n$ lies in $\ker(\phi_r)$ if and only if the virtual undercrossing action for $\beta$ satisfies $\beta \cdot \vec 0 = \vec 0$. If $\beta$ is moreover a pure braid, then $\beta \cdot \vec x = \vec x$ for $\vec x \in (\Z/r\Z)^{n+1}$ arbitrary.
\end{lemma} 
\begin{proof}
    As noted above, applying the virtual undercrossing procedure on a braid $\beta$ to $\vec 0$ yields $v(\beta)$. By \Cref{lemma:equivalent}, $v(\beta) = \vec 0$ if and only if $\phi_r(\beta) = \vec 0$. If $\beta$ is any pure braid, then $M(\beta) = I_{n+1}$, and so $\beta \cdot \vec x = M(\beta) \vec x + v(\beta) = \vec x + v(\beta)$. Thus if $\beta \in \ker(\phi_r)$ is pure, $\beta \cdot \vec x = \vec x$ as claimed. 
\end{proof}

\subsection{The factorization algorithm}\label{subsection:factorize}

As discussed in the section outline above, the factorization algorithm in this section gives a method for expressing an element of $\ker(\phi_r)$ in $\Gamma_n^r$ when it is supported on a small number of strands. Our algorithm will require that the supporting subdisk have a particularly simple form which we call a {\em standard embedding}; we begin with this definition.

\begin{definition}[Standard arc]\label{definition:standardarc}
    Let $D_n$ denote the disk with $n$ marked points. An embedded arc $\alpha \subset D_n$ with endpoints at distinct marked points is {\em standard} if it is contained entirely in the lower half-disk. For each pair of marked points $i,j$, there is a unique isotopy class of standard arc connecting $i$ and $j$, which is denoted $\alpha_{ij}$.
\end{definition}

\begin{definition}[Standard embedding]
\label{definition:standardembed}
Let $D_n$ denote the disk with $n$ marked points. An embedding $i: D_s \to D_n$ sending marked points to marked points is {\em standard} if it can be represented as a regular neighborhood of a union of standard arcs which are disjoint except at endpoints. 
\end{definition}

An example of a standard embedding is depicted in \Cref{figure:embed}.
\begin{figure}[ht]
\labellist
\tiny
\endlabellist
\includegraphics[scale=0.6]{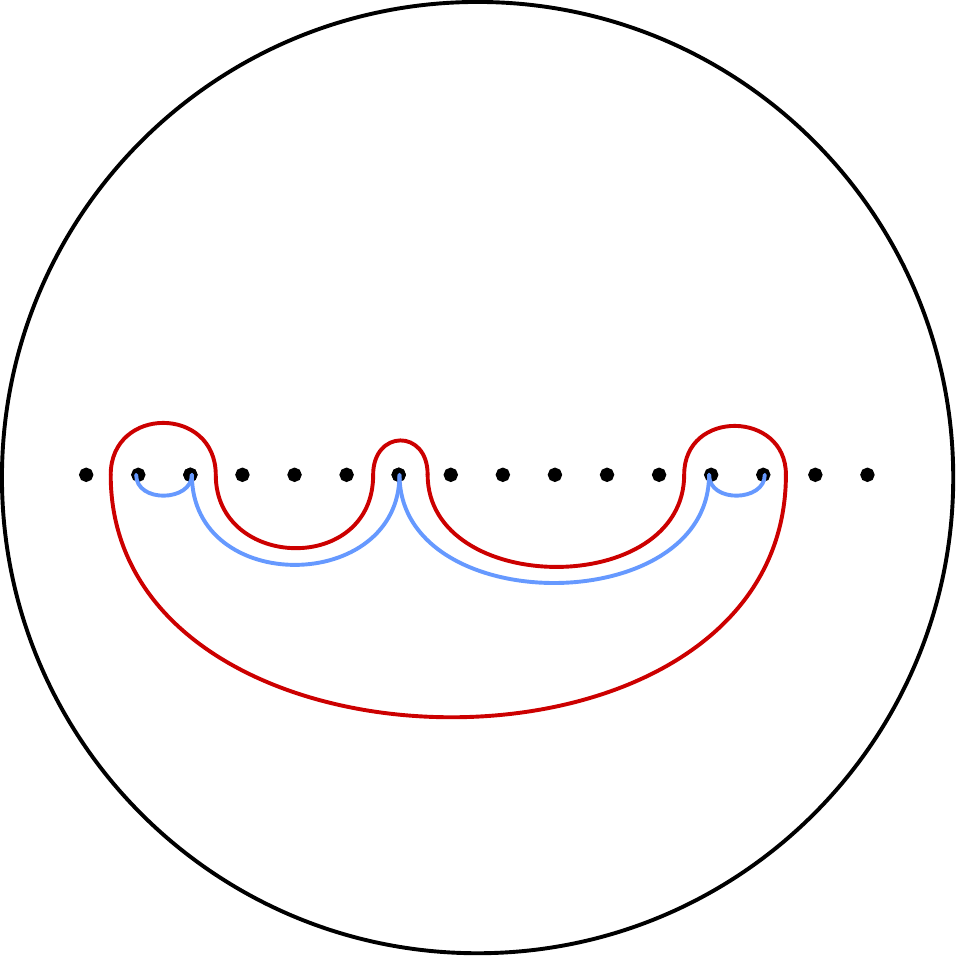}
\caption{The red disk is a standard embedding of $D_5$, as it is a regular neighborhood of the four standard arcs shown in blue.} 
\label{figure:embed}
\end{figure}

\begin{lemma}[Factorization algorithm]
    \label{prop:factorize}
    Let $n \ge rs$. Let $i: D_s \to D_n$ be a standard embedding, and let $i_*: B_s \to B_n$ denote the corresponding inclusion of braid groups. Then there is a containment $i_*(PB_s) \cap \ker(\phi_r) \le \Gamma_n^r$.
\end{lemma}

\begin{figure}[ht]
\labellist
\tiny
\pinlabel $\red{0}$ [b] at 0.00 286.27
\pinlabel $\red{0}$ [b] at 28.34 286.27
\pinlabel $\red{0}$ [b] at 73.69 286.27
\pinlabel $\red{0}$ [b] at 104.87 225
\pinlabel $\red{1}$ [b] at 138.88 225
\pinlabel $\red{1}$ [b] at 174.90 225
\pinlabel $\red{1}$ [b] at 209.74 156
\pinlabel $\red{0}$ [b] at 240.92 156
\pinlabel $\red{1}$ [b] at 279.77 156
\pinlabel $\red{1}$ [b] at 317.45 90
\pinlabel $\red{1}$ [b] at 348.63 90
\pinlabel $\red{0}$ [b] at 374.14 90
\pinlabel $\red{0}$ [b] at 425.16 51
\pinlabel $\red{0}$ [t] at 450.67 50
\pinlabel $\red{0}$ [b] at 479.01 51
\endlabellist
\includegraphics[scale=0.9]{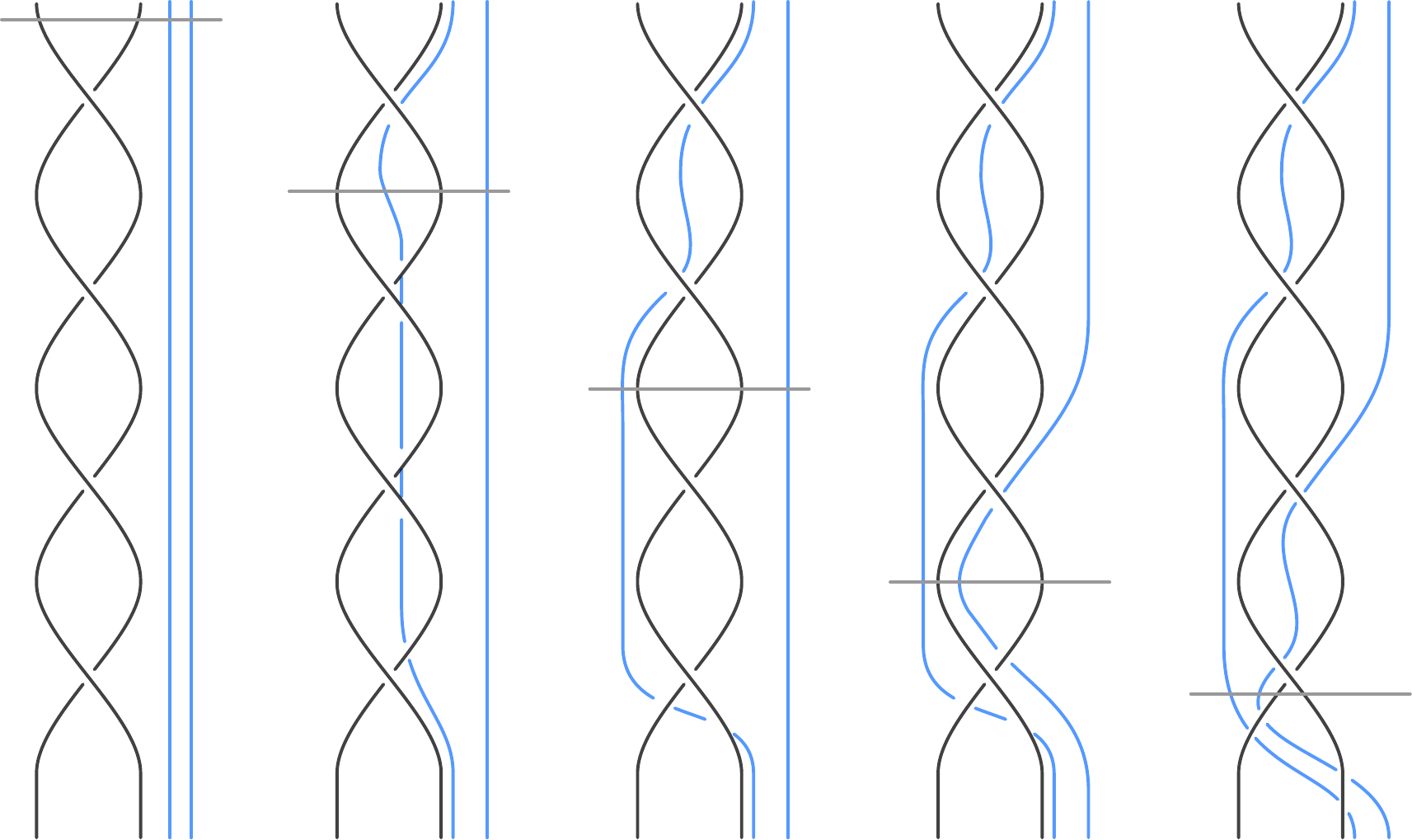}
\caption{The factorization algorithm applied in the case $r = 2$ to $\sigma_1^4$. Reading left to right, we work our way down through the crossings, borrowing blue strands so that successive crossings have a total of $r =2$ strands crossing under. Since $\phi_2(\sigma_1^4) = \vec 0$, after performing this procedure at all four crossings, the number of blue strands in each position is $0 \pmod 2$, and they can be passed back under in pairs, preserving the property that every overcrossing has an even number of strands passing underneath. Thus the algorithm produces the factorization $\sigma_1^4 = (\sigma_1 \sigma_2)^2(\sigma_2 \sigma_3)^2(\sigma_1 \sigma_2)^{-1} (\sigma_2 \sigma_3)^{-1}$. The red numbers indicate the counts of blue strands in the indicated positions at the indicated levels; notice that they record the values of $v(\sigma_1^k)$ for $k = 0,\dots,4$.} 
\label{figure:example}
\end{figure}

\begin{proof}
 We begin with an important special case, when $D_s$ is the standard disk consisting of the first $s$ points. This will serve to illustrate all of the key ideas of the argument. Then we will discuss the modifications necessary to apply in the general case. 

\para{Special case: first $s$ strands} While reading this portion of the argument, the reader is invited to consult the worked example demonstrated in \Cref{figure:example}. Let $\beta \in PB_s \cap \ker(\phi_r)$ be given. We view this as a braid $\alpha$ on $s$ strands juxtaposed with a trivial braid $\tau$ on $n-s \ge (r-1)s$ strands lying to the right. The key idea is to treat the strands of $\tau$ as the virtual strands in the virtual undercrossing procedure. Accordingly, we will depict the strands of $\alpha$ as black, and those of $\tau$ as blue, as in our discussion of virtual undercrossings above.

Recall (\Cref{remark:undercrossing}) that a basic $(r+1)^{st}$ twist $\Sigma_{i;r}$ consists of a single strand passing over $r$ strands, so that in order to exhibit $\beta$ as an element of $\Gamma_n^r$, it suffices to factor $\beta$ so that all crossings have this form. To perform the factorization, we will isotope the strands of $\tau$, moving them to the left so that each overcrossing in $\alpha$ has $0 \pmod r$ total strands (black and blue) passing underneath. 

In carrying this factorization out, we will make use of the following operation. Given a braid $\beta$, obviously the product $\beta \Sigma_{i;r}$ lies in the same left coset of $\Gamma_n^r$ as $\beta$. Graphically, $\beta \Sigma_{i;r}$ is obtained from $\beta$ by taking the packet of $r$ consecutive strands from $i+1$ to $i+r$ and passing them one unit to the left under the $i^{th}$ strand, the latter of which moves over $r$ units to the right. We call this procedure {\em passing a packet} to the left; evidently there is also the analogous move of passing a packet to the right, corresponding to right-multiplication by $\Sigma_{i;r}^{-1}$. Likewise, we do not change the left coset by passing packets of $r$ strands at the top of the braid.

Before presenting the algorithm, we make one final observation. Suppose we are given a particular braid diagram for $\beta$ (not just its isotopy class). As usual, we think of the strands of $\alpha$ as black and the strands of $\tau$ as blue. At any vertical level where no two strands of $\beta$ (black or blue) cross, we have a well-defined count of the number of blue strands in between each adjacent pair of black strands, giving us an integer vector $\vec v$ with $s+1$ entries. We call the space between strands $i$ and $i+1$ of $\alpha$ as {\em the $i^{th}$ position}. If the blue strands of $\tau$ are isotoped so as to conform to the conventions of the virtual undercrossings procedure (as described in \Cref{figure:undercrossing}), the reduction of $\vec v \pmod r$ is equal to $v(\alpha')$, where $\alpha'$ is the portion of $\alpha$ from the top down to the specified vertical level.\\

We now explain the factorization algorithm. Express $\alpha$ as a product of the standard generators of $B_s$, and suppose $\alpha$ begins with $\sigma_i^{\epsilon}$ with $\epsilon = \pm 1$. To begin the factorization, pass a packet of the first $r$ strands of $\tau$ to the left; in the case of $\epsilon = 1$, pass these to the $(i+1)^{st}$ position, and if $\epsilon = -1$, pass these to the $(i-1)^{st}$. Pass $r-1$ of these under the overcrossing and the remaining strand straight down, exactly as illustrated in \Cref{figure:undercrossing}. We call this process {\em resolving a crossing}.

Now repeat this procedure for the remaining crossings of $\alpha$: pass packets of $r$ blue strands from $\tau$ to the relevant position, and borrow the necessary number of strands so as to create an undercrossing by $r$ strands. If ever there are $r$ or more blue strands in a single position after resolving a crossing, pass them in multiples of $r$ all the way back to the right. 

We must verify that as long as there are at least $(r-1)s$ blue strands, it is always possible to pass a packet of $r$ strands from the right over to the location of the overcrossing so as to facilitate a borrowing. Borrowing from the right is necessary only when the number of blue strands in consecutive positions is strictly less than the $r-1$ needed to ensure that $r$ strands pass under the given overcrossing, i.e. there are consecutive strand counts $x_i, x_{i+1}$ for which $x_i + x_{i+1} \le r -2$. Since we have passed packets of $r$ strands to the right (i.e. to $x_n$) whenever possible, each of the remaining $s-2$ components $x_j$ for $0 \le j \le n-1$ has $x_j \le r-1$. Altogether then, in this situation, we have
\[
\sum_{i = 0}^{n-1}x_i \le (s-2)(r-1) + r-2 = (r-1)s - r,
\]
so that $x_n \ge r$ as was to be shown.

By \Cref{lemma:kernel0}, at the conclusion of this process, the number of blue strands mod $r$ in each position is equal to the corresponding component of $v(\alpha) = \vec 0$. As we have methodically passed packets of $r$ blue strands to the right whenever possible, this shows that in fact there are no blue strands in between the black strands of $\alpha$. In other words, the resuling braid diagram is isotopic to the original juxtaposition of $\alpha$ and $\tau$. On the other hand, we have isotoped the blue strands of $\tau$ so that at every overcrossing, there are $0 \pmod r$ strands passing underneath, exhibiting $\beta$ as a product of $(r+1)^{st}$ twists as required.

\begin{figure}[ht]
\labellist
\tiny
\pinlabel $\red{1}$ [b] at -5.00 286.27
\pinlabel $\red{1}$ [b] at 28.34 286.27
\pinlabel $\red{0}$ [b] at 60.69 286.27
\pinlabel $\red{0}$ [b] at 98.87 240
\pinlabel $\red{0}$ [b] at 138.88 240
\pinlabel $\red{0}$ [b] at 164.90 240
\pinlabel $\red{0}$ [b] at 209.74 146
\pinlabel $\red{1}$ [b] at 233.92 146
\pinlabel $\red{1}$ [b] at 279.77 146
\pinlabel $\red{1}$ [b] at 323.45 70
\pinlabel $\red{0}$ [b] at 348.63 70
\pinlabel $\red{1}$ [b] at 369.14 70
\pinlabel $\red{1}$ [b] at 427.16 10
\pinlabel $\red{1}$ [b] at 450.67 10
\pinlabel $\red{0}$ [b] at 479.01 10
\endlabellist
\includegraphics[scale=0.9]{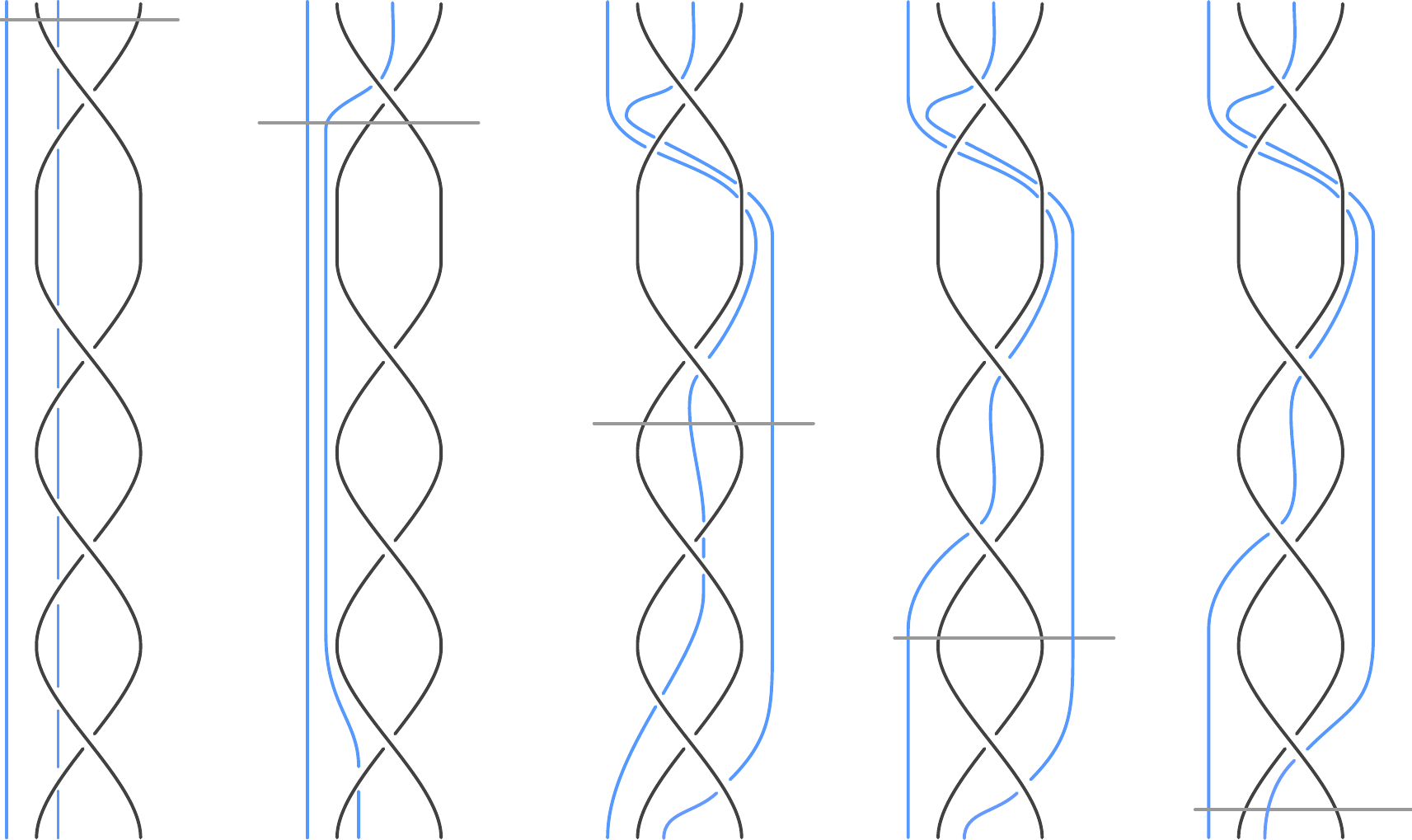}
\caption{The factorization algorithm as applied to a braid under a standard embedding ($r = 2$). As before, the method is to work down from the top, resolving crossings so as to have $r$ strands crossing under each overcrossing. Note that after each crossing is resolved, the count of blue strands in each position is given by $\alpha' \cdot \vec x$, where $\vec x = (1,1,0)$ is the count of blue strands at the top of the diagram, $\alpha'$ is the initial segment of $\alpha$ to the given level, and the action is given by \eqref{equation:beta}.} 
\label{figure:example2}
\end{figure}
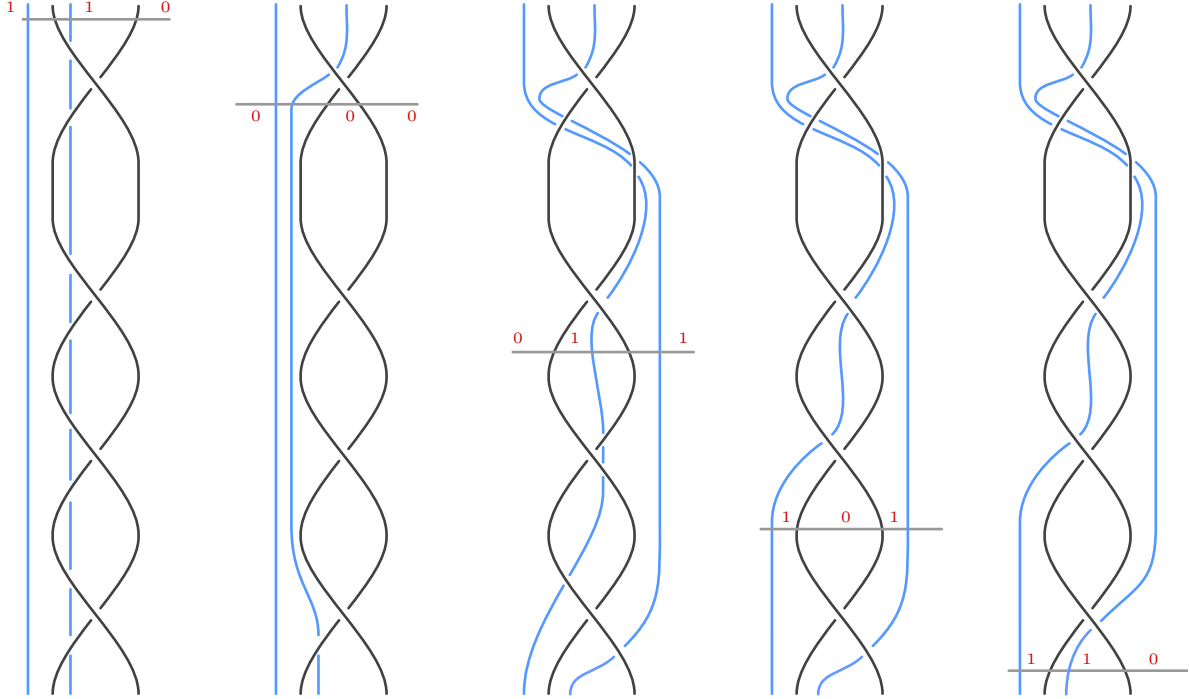

\para{General case: arbitrary standard embedding} The reader is now invited to consult \Cref{figure:example2}. Let $i: D_s \to D_n$ be a standard embedding, and let $\beta \in i_*(PB_n) \cap \ker(\phi_r)$ be given. As $i$ is standard, we can represent $\beta$ as a juxtaposition of a braid $i_*(\alpha)$ for $\alpha \in PB_s$ in black {\em on top of} a trivial braid of $n-s$ strands in blue. In the language established above, the only difference between this setting and the special case above is that here we begin the factorization algorithm with blue strands in arbitrary positions, not with all $n-s$ strands in position $s$ as above. We proceed as before, working our way down from the top, resolving crossings by borrowing blue strands. The analysis above applies verbatim to show that when $n \ge rs$, there are sufficiently many blue strands available to make borrowing possible. It remains to be shown that the blue strands return to their original positions after all of the crossings of $\alpha$ are resolved. Recalling the hypothesis that $\beta$ be a {\em pure} braid, this now follows from \Cref{lemma:kernel0}. 
\end{proof}

\subsection{Factoring general braids}\label{subsection:endgame}
In this section, we conclude the proof of \Cref{mainthm:monodromy}. The main technical result is \Cref{lemma:factorpure}, which establishes the containment $\ker(\phi_r) \cap PB_n \le \Gamma_n^r$. From there, the full containment (\Cref{theorem:gammakernel}) and the proof of \Cref{mainthm:monodromy} are relatively easy. 

To begin the analysis of $\ker(\phi_r) \cap PB_n$, we investigate the restriction of $\phi_r$ to $PB_n$.
\begin{lemma}
    \label{lemma:phionpb}
    The restriction of $\phi_r$ to $PB_n$ is a genuine homomorphism $\phi_r: PB_n \to (\Z/r\Z)^n$, given on the standard generators $A_{ij}$ of $PB_n$ via
    \[
    \phi_r(A_{ij}) = e_i + e_j.
    \]
\end{lemma}
\begin{proof}
    Since the action of $B_n$ on $(\Z/r\Z)^n$ factors through the quotient $B_n \to S_n$, it follows that the restriction of $\phi_r$ to $PB_n$ is a homomorphism. We evaluate
    \[
    \phi_r(A_{12}) = \phi_r(\sigma_1^2) = \phi_r(\sigma_1) + \sigma_1 \cdot \phi_r(\sigma_1) = (1 + (12))\cdot e_2 = e_1 + e_2.
    \]
    The formula
    \[
    \phi_r(ghg^{-1}) = g \cdot \phi_r(h)
    \]
    for $g \in B_n$ and $h \in PB_n$ is readily seen to hold, from which the expression $\phi_r(A_{ij}) = e_i + e_j$ follows.
\end{proof}

\begin{lemma}\label{lemma:factorpure}
    Let $n \ge 3$ and $r \ge 2$ be given, and let $d \in \{0,1,2\}$ be the remainder of $n/3$. Then for $n \ge n_0(r,d)$ (where $n_0(r,d)$ is defined as in \eqref{equation:nord}), there is a containment $PB_n \cap \ker(\phi_r) \le \Gamma_n^r$.
\end{lemma}
\begin{proof}
    Let 
    \begin{equation}\label{equation:w}
    w = \prod_{k = 1}^m A_{i_k j_k}^{\epsilon_k} \in PB_n \cap \ker(\phi_r)
    \end{equation}
    be given. By hypothesis,
    \[
    \phi_r(w) = \sum_{k = 1}^m \epsilon_k (e_{i_k}+e_{j_k}) = \vec 0 \in (\Z/r\Z)^n.
    \]

    To express $w \in \Gamma_n^r$, we will exploit the factorization algorithm (\Cref{prop:factorize}) to rewrite the initial segment of $w$ as a product of commuting elements of small support which has the same value under $\phi_r$, removing initial segments that lie in $\ker(\phi_r)$ whenever possible. This will take slightly different forms in the regimes $r = 2$, $r$ odd, and $r \ge 4$ even; we begin with the case $r = 2$ since it is the simplest and will serve to illustrate the essential idea.

    \para{The case $\boldmath{r = 2}$} Assume $r = 2$, and write
    \[
    w = A_{i_1 j_1}^{\epsilon_1} A_{i_2 j_2}^{\epsilon_2} w'
    \]
    with $w'$ given as the product defining $w$ in \eqref{equation:w} above but starting at $k = 3$.
    
    There are three possibilities for the size of the set $\{i_1,j_1, i_2, j_2\}$. Suppose first that $\{i_1, j_1\} = \{i_2, j_2\} = \{i,j\}$, in which case necessarily $\epsilon_1 = \epsilon_2$. Suppose for simplicity $\epsilon_1 = \epsilon_2 = 1$; the argument in the other case is analogous. Here, $A_{ij}^2 \in \ker(\phi_2)$. We note that $A_{ij}^2$ is a pure braid in $\ker(\phi_2)$ under a standard embedding of a disk with two marked points, and as $n \ge 8$ by hypothesis, we can apply the factorization algorithm \Cref{prop:factorize} to express $A_{ij}^2 \in \Gamma_n^2$. Thus in this case, we can write $w = \gamma w'$ with $\gamma \in \Gamma_n^2$, proceeding in turn to factorize $w'$.

    Suppose next that $i_1 = i_2$ but $j_1 \ne j_2$. In this case, we have that 
    \[
    \phi_2(A_{i_1j_1}^{\pm 1} A_{i_2 j_2}^{\pm 1}) = e_{j_1} + e_{j_2} = \phi_2(A_{j_1 j_2}),
    \]
    so that
    \[
    A_{i_1j_1}^{\pm 1} A_{i_2 j_2}^{\pm 1} A_{j_1 j_2}^{-1} \in \ker(\phi_2).
    \]

    This element is again in the image of a standard embedding, so that we can apply the factorization algorithm (\Cref{prop:factorize}) to express $A_{i_1j_1}^{\pm 1} A_{i_2 j_2}^{\pm 1} A_{j_1 j_2}^{-1}$ as an element of $\Gamma_n^2$; the disk is standard and has four marked points, and hence this is possible for $n \ge 8$. Thus, by left-multiplying by an element of $\Gamma_n^2$ we can replace the initial segment $A_{i_1j_1}^{\pm 1} A_{i_2 j_2}^{\pm 1}$ with the initial segment $A_{j_1 j_2}$, which has smaller support. Similar arguments apply to the various cases when $\abs{\{i_1,j_1,i_2,j_2\}} = 3$.

    The remaining possibility is that $\abs{\{i_1,j_1,i_2,j_2\}} = 4$. In this case, we replace the initial segment $A_{i_1j_1}^{\pm 1} A_{i_2 j_2}^{\pm 1}$ with the segment $A_{a,b} A_{c,d}$, where $\{i_1,j_1, i_2, j_2\} = \{a,b,c,d\}$ and $a < b < c < d$. Note in particular that the elements $A_{a,b}, A_{c,d}$ of the initial segment commute, and that the pair of elements $A_{a,d} A_{b,c}$ also commute; we say that the former pair is {\em un-nested} and the latter {\em nested}.

    We continue in this way, expressing
    \[
    w = \gamma A_{i_p, j_p}^{\pm 1} \prod_{k = p+1}^m A_{i_k j_k}^{\epsilon_k}
    \]
    with $\gamma$ a product of pairwise un-nested commuting generators of $PB_n$. The support of the element $A_{i_p, j_p}^{\pm 1}$ intersects the support of $0, 1,$ or $2$ of the elements of $\gamma$. If it intersects zero, it may be nested with up to one. This can be resolved by pulling these two commuting elements to the front of $\gamma$ and replacing them as above with their un-nested counterpart. If it intersects one, these two elements can likewise be moved to the front of $\gamma$ and resolved into one or two basic elements as above. Finally suppose it intersects the support of two, say $A_{i_1,j_1}$ and $A_{i_2, j_2}$; by the non-nestedness hypothesis, $i_1 < j_1 < i_2 < j_2$. For $A_{i_p, j_p}^{\pm 1}$ to intersect both, we must have
    \[
    i_1 \le i_p \le j_1 \mbox{ and } i_2 \le j_p \le j_2.
    \]
    Thus the product $A_{i_1,j_1}A_{i_2, j_2}A_{i_pj_p}$ is supported on a standardly-embedded disk of up to six elements. As we are only assuming $n \ge 8$ and the factorization algorithm (\Cref{prop:factorize}) requires $n \ge 2s = 12$ strands to factor an element supported on six strands, we first rewrite $A_{i_1,j_1}A_{i_2, j_2} = A_{i_1 j_2} A_{i_2 j_1}$. If $i_1 < i_p < j_1 < i_2 < j_p < j_2$, we can then rewrite $A_{i_1 j_2} A_{i_p j_p}^{\pm}$ as the un-nested pair $A_{i_1 i_p} A_{j_p j_2}$ which is then un-nested with $A_{j_1 i_2}$. If $i_1 = i_p$, then the initial segment $A_{i_1 j_2} A_{i_1 j_p}$ can be replaced with $A_{j_2 j_p}$; the remaining cases where the total support is five strands can be handled analogously. Finally, if $A_{i_p j_p}^{\pm 1}$ intersects both $A_{i_1 j_1}$ and $A_{i_2 j_2}$ but the total number of strands is four, then this can be rewritten directly via the factorization algorithm without any need for initial re-writing.

    Altogether then, this process converts an arbitrary word $w \in PB_n \cap \ker(\phi_2)$ into a product of pairwise un-nested and commuting generators $A_{ij}$. Since $\phi_2(w) = 0$ by hypothesis, this implies that each $A_{ij}$ appears an even number of times; these can then be successively removed from $w$ by means of the factorization algorithm (\Cref{prop:factorize}).\\

    \para{The case $\boldmath{r\ge 3}$ odd} In broad outline, we proceed in the same way as in the case $r = 2$. Given $w \in \PB_n \cap \ker(\phi_r)$ as in \eqref{equation:w}, we rewrite the initial segment of $w$ as a set of elements with disjoint and small support. Whereas in the case $r = 2$ these elements were the generators $A_{ij}$ of $PB_n$, here we will use elements $A_{i;a,b,c}$ which we proceed to define.

    Let $a, b, c \in \Z/r\Z$ be given. Set $p = (r+1)/2$, and define
    \[
    A_{i;a,b,c} = A_{i,i+1}^{p(a+b-c)} A_{i,i+2}^{p(a-b+c)} A_{i+1,i+2}^{p(-a + b + c)}.
    \]
    Observe that
    \[
    \phi_r(A_{i;a,b,c}) = a e_i + b e_{i+1} + c e_{i+2}.
    \]
    Also note that $A_{i;a,b,c}$ and $A_{j;a',b',c'}$ have disjoint support whenever $\abs{i - j} \ge 3$.
    
    Divide the strands from $1$ to $n$ into groups of three; the last group will contain $3,4,$ or $5$ depending on the value of $d$, i.e. the remainder of $n \pmod 3$. We extend the definition of $A_{1+3k; a,b,c}$ to this last group by taking $A_{1+3k; a,b,c,d,e} = A_{1+3k,a,b,c}A_{3 + 3k,0,d,e}$ in the case $d = 2$ and similarly for $d=1$; to simplify notation we will tacitly understand that the last $A_{i;a,b,c}$ may be of this form. Suppose we have a partial factorization
    \[
    w = \gamma A_{i_p,j_p}^{\epsilon_p} w',
    \]
    where $\gamma$ is a product of elements of the form $A_{1+3k;a,b,c}$. Then $A_{i_p,j_p}^{\epsilon_p}$ intersects the support of at most two such elements, and altogether the product of these three elements is supported on a standardly-embedded disk with at most $6+d$ punctures. Pulling these to the front of $\gamma$, since we assume $n \ge n_0(r,d) = (6+d)r$, we apply the factorization algorithm (\Cref{prop:factorize}) and replace this with a product of up to two elements of the form $A_{1+3k;a,b,c}$ with the same $\phi_r$-value.

    After completing this process, we have factored $w = \gamma$ into a product of elements of the form $A_{1+3k; a,b,c}$. As $\phi_r(w) = \vec 0$ by hypothesis, it follows that each $\phi_r(A_{1+3k;a,b,c}) = \vec 0$ as well. Applying the factorization algorithm to each of these in turn, we express $w = \gamma$ as an element of $\Gamma_n^r$.

    \para{The case $\boldmath{r \ge 4}$ even} In this last case, we combine the methods of the previous two. Again, the objective is to factor initial segments of $w$ into disjoint elements of small support with the same value of $\phi_r$. Like in the case of $r$ odd, we partition the strands into groups of three and attempt to factor the initial segment into elements supported on these groups. But unlike this case, there is a parity phenomenon to keep track of, which will require us to {\em link} two such groups if the parity of $\phi_r$ on each is odd. 

    We define $A_{i;a,b,c}$ analogously as above, this time subject to the requirement that $a+b+c$ be even; under this hypothesis, each of the integers $\pm a \pm b \pm c$ is even and so can be divided by two in the exponent. As before, the last $A_{1+3k; a,b,c}$ may actually be supported on up to five strands. Given an initial segment $\gamma \in B_n$ of $w$, we say that a group of integers $1+3k, 2+3k, 3+ 3k$ is {\em even} if the sum of the coefficients of $\phi_r(\gamma)$ on $e_{1+3k}, e_{2+3k}, e_{3 + 3k}$ is even, and {\em odd} otherwise. Observe that there is always an even number of odd groups, since each $A_{ij}^{\pm 1}$ changes the parity of either zero or two groups.

    We now describe the structure of the initial segment we will construct. We will express $w = \gamma w'$ where $\gamma$ is a product of $A_{1+3k;a,b,c}$ over all even groups, along with products of the form 
    \[
    A_{1 + 3k; a,b,c} A_{1 + 3k';a',b',c'} A_{i,j}, 
    \]
    where the groups starting at $1+3k$ and $1+3k'$ are odd, and where the first group contains $i$ and the second contains $j$. We moreover impose the condition that there are no odd groups in between $1+3k$ and $1+3k'$. In this way, the structure of the supports mimics that in the case $r = 2$: they are disjoint, un-nested, and supported on standardly-embedded disks. 

    Given a partial factorization
    \[
    w = \gamma A_{i_p j_p}^{\epsilon_p} w'
    \]
    of this form, we consider the various possibilities for how the support of $A_{i_p j_p}^{\epsilon_p}$ intersects the supports of the elements in $\gamma$. This exactly mirrors the analysis carried out in the case $r = 2$, but this time we apply the factorization algorithm to elements supported on up to $12+d$ strands, in the case where we need to convert between nested and un-nested factorizations on two pairs of odd groups, one of which contains the exceptional group of $3+d$ elements.
\end{proof}

\begin{theorem}
    \label{theorem:gammakernel}
    Let $n_0(r,d)$ and $d$ be given as in \Cref{lemma:factorpure}. Then for $n \ge n_0(r,d)$, there is an equality
    \[
    \ker(\phi_r) = \Gamma_n^r,
    \]
    i.e. $\ker(\phi_r)$ is generated by the finite set of basic $(r+1)^{st}$ twists.
\end{theorem}

\begin{proof}
    The bulk of the work has been carried out above in \Cref{lemma:factorpure}, which establishes the containment
    \[
    \PB_n \cap \ker(\phi_r) \le \Gamma_n^r
    \]
    in the range $n \ge n_0(r,d)$. Conversely, it is easy to verify that
    \[
    \phi_r(\Sigma_{i;r}) = 0,
    \]
    so that
    \[
    \Gamma_n^r \le \ker(\phi_r).
    \]
    It remains only to show that the images of $\ker(\phi_r)$ and $\Gamma_n^r$ in $S_n$ coincide; denote these subgroups of $S_n$ by $\bar{\ker(\phi_r)}$ and $\bar{\Gamma_n^r}$, respectively. 

    The basic $(r+1)^{st}$ twists that generate $\Gamma_n^r$ are sent to the $r+1$-cycles 
    \[(1\dots r+1), (2 \dots r+2),\dots, (n-r-1\dots n)
    \]
    in $S_n$. Thus
    \[
    (1\ r+1\ r+2) = (1 \dots r+1)^{-1} (2\dots r+2) \in \bar{\Gamma_n^r},
    \]
    and hence also
    \[
    (1\ 2\ 3) = (2\dots r+2)^2 (1\ r+1\ r+2) (2 \dots r+2)^{-2} \in \bar{\Gamma_n^r}.
    \]
    As also $\Gamma_n^r$ contains the element $\Sigma_{1;n-1} = \sigma_1 \dots \sigma_{n-1}$, the image $\bar{\Gamma_n^r}$ contains the cyclic  permutation $(1 \dots n)$. Conjugating $(1\ 2\ 3)$ by this, it follows that $\bar{\Gamma_n^r}$ contains all $3$-cycles of the form $(i\ i+1\ i+2)$. This is well-known to generate the alternating group $A_n$. We conclude that
    \[
    A_n \le \bar{\Gamma_n^r}
    \]
    for all $n, r$. For $r$ odd, the $r+1$-cycles are odd permutations and are even otherwise, from which it follows that 
    \[
    \bar{\Gamma_n^r} = \begin{cases}
        S_n & r \mbox{ odd}\\
        A_n & r \mbox{ even}.
    \end{cases}
    \]

It remains only to show that $\bar{\ker(\phi_r)} \le A_n$ when $r$ is even. To see this, we recall that $\phi_r$ can be viewed as the homomorphism 
\[
\phi_r: B_n \to S_n \ltimes (\Z/r\Z)^n
\]
which sends $\sigma_i$ to the pair $((i\ i+1), e_{i+1})$. Let $\sgn: S_n \to \Z/2\Z$ denote the sign homomorphism, and let $s: (\Z/r\Z)^n \to \Z/2\Z$ be the reduction mod $2$ of the sum-of-coefficients map. Then
\[
\sgn + s: S_n \ltimes (\Z/r\Z)^n \to \Z/2\Z
\]
is a surjective homomorphism, and the composition $(\sgn+s) \circ \phi_r$ is identically zero (being zero on each generator of $B_n$ by above), from which it follows that $\bar{\ker(\phi_r)} \le A_n$ as was to be shown.
\end{proof}

\begin{proof}[Proof of \Cref{mainthm:monodromy}]
    \Cref{lemma:monodromycontain} establishes the containment
    \[
    B_n[\kappa] \le \ker(\phi_r),
    \]
    \Cref{lemma:monodromygens} shows that
    \[
    \Gamma_n^r \le B_n[\kappa],
    \]
    and \Cref{theorem:gammakernel} shows the equality
    \[
    \Gamma_n^r = \ker(\phi_r)
    \]
    in the range $n \ge n_0(r,d)$.
\end{proof}

\bibliographystyle{alpha}
\bibliography{bib}

\end{document}